\documentclass{amsart}%
\usepackage{amssymb}
\usepackage{amsmath}
\usepackage{amsfonts}
\usepackage{graphicx}
\usepackage[all,cmtip]{xy}%
\newtheorem{theorem}{Theorem}
\theoremstyle{plain}

\newtheorem{corollary}{Corollary}

\newtheorem{definition}{Definition}
\newtheorem{example}{Example}

\newtheorem{lemma}{Lemma}

\newtheorem{proposition}{Proposition}
\newtheorem{remark}{Remark}

\numberwithin{equation}{section}
\begin{document}
\title[Tate's central extension]{Ad\`{e}le residue symbol and Tate's central extension for multiloop Lie algebras}
\author{Oliver Braunling}
\address{Fakult\"{a}t f\"{u}r Mathematik\\
Universit\"{a}t Duisburg-Essen\\
Thea-Leymann-Stra\ss e 9\\
45127 Essen\\
Germany}
\email{oliver.braeunling@uni-due.de}
\urladdr{http://www.esaga.uni-due.de/oliver.braeunling/}
\thanks{This work has been supported by the Alexander von Humboldt Foundation and the DFG SFB/TR 45 \textquotedblleft Periods,
moduli spaces and arithmetic of algebraic varieties\textquotedblright.}
\subjclass[2000]{Primary 17B56, 17B67; Secondary 32A27}

\begin{abstract}
We generalize the linear algebra setting of Tate's central extension to
arbitrary dimension. In general, one obtains a Lie $(n+1)$-cocycle. We compute
it to some extent. The construction is based on a Lie algebra variant of
Beilinson's adelic multidimensional residue symbol, generalizing Tate's
approach to the local residue symbol for $1$-forms on curves.

\end{abstract}
\maketitle

Firstly, recall that to every Lie algebra $\mathfrak{g}$ one can associate its
loop Lie algebra $\mathfrak{g}[t^{\pm}]$. Iterating this construction, we
obtain so-called \textit{multiloop Lie algebras}, $\mathfrak{g}[t_{1}^{\pm
1},\ldots,t_{n}^{\pm1}]$.\newline To begin with, we show that various classes
of interesting multiloop\ Lie algebras can all be embedded into a large
(infinite-dimensional) Lie algebra:

\begin{theorem}
\label{intro_Thm1UniversalTateCocycle}Let $k$ be a field and $n\geq1$. There
is a universal Lie algebra $\mathfrak{G}$ naturally containing the following:

\begin{enumerate}
\item the abelian Lie algebra $k[t_{1}^{\pm1},\ldots,t_{n}^{\pm1}]$,

\item Lie algebras of derivations, e.g. spanned by%
\[
t_{1}^{s_{1}}\cdots t_{n}^{s_{n}}\partial_{t_{i}},\qquad\text{(acting on
}k[t_{1}^{\pm1},\ldots,t_{n}^{\pm1}]\text{)}%
\]

\item for any finite-dimensional simple Lie algebra $\mathfrak{g}$ the
multiloop algebra%
\[
\mathfrak{g}[t_{1}^{\pm1},\ldots,t_{n}^{\pm1}]\text{.}%
\]

\end{enumerate}

The universal Lie algebra $\mathfrak{G}$ has a canonical Lie $(n+1)$-cocycle
$\phi\in H^{n+1}(\mathfrak{G},k)$. For $n=1$ this cocycle determines a central
extension\ (known as \emph{Tate's central extension})%
\[
0\longrightarrow k\longrightarrow\widehat{\mathfrak{G}}\longrightarrow
\mathfrak{G}\longrightarrow0
\]
and the pullback of it to one of the above types of subalgebras yields (respectively)

\begin{enumerate}
\item the Heisenberg algebra,

\item the Virasoro algebra,

\item the affine\ Lie algebra $\widehat{\mathfrak{g}}$ associated to
$\mathfrak{g}$.
\end{enumerate}
\end{theorem}

This will be stated in more detail and proven in
\S \ref{section_ConcreteFormalism}. It is not at all surprising that some Lie
algebras can be embedded into larger ones. The interesting fact is that there
is such a Lie algebra which carries a canonical cocycle, inducing the ones
defining all these classical central extensions. For $n=1$ the above is
well-known, see for example \cite[\S 2.1]{BFM_Conformal}. For $n=1,2$ see
\cite{Frenkel08092011}. In the language of the latter, $\mathfrak{G}$ is an
example of a \textquotedblleft master Lie algebra\textquotedblright%
.\medskip\newline We are interested in the nature of $\phi$ for $n>1$ -- even
if such cocycles cannot be interpreted as a central extension anymore (we get
crossed modules, etc.). Indeed, they are meaningful, as we shall see.\medskip

A key point of this text is the actual computation of $\phi$ (with a slight limitation):

\begin{theorem}
\label{intro_Thm2CocycleFormula}The cocycle $\phi\in H^{n+1}(\mathfrak{G},k)$
is given explicitly by%
\begin{align*}
&  \phi(f_{0}\wedge f_{1}\wedge\ldots\wedge f_{n})\\
&  =\operatorname*{tr}%
{\textstyle\sum_{\pi\in\mathfrak{S}_{n}}}
\operatorname*{sgn}(\pi)%
{\textstyle\sum_{\gamma_{1}\ldots\gamma_{n}\in\{\pm\}}}
\left(  -1\right)  ^{\gamma_{1}+\cdots+\gamma_{n}}(P_{1}^{-\gamma_{1}%
}\operatorname*{ad}(f_{\pi(1)})P_{1}^{\gamma_{1}})\\
&  \qquad\qquad\cdots(P_{n}^{-\gamma_{n}}\operatorname*{ad}(f_{\pi(n)}%
)P_{n}^{\gamma_{n}})f_{0}\text{,}%
\end{align*}
whenever $f_{0}\otimes f_{1}\wedge\ldots\wedge f_{n}$ is already a
$\mathfrak{g}$-valued Lie cycle. The $P_{1}^{+},\ldots,P_{n}^{+}$ refer to
certain commuting idempotents (see \S \ref{section_CubeComplex} for details).
\end{theorem}

The proof and details regarding the $P_{i}^{\pm}$ can be found in
\S \ref{section_ConcreteFormalism}. Effectively, we compute the composition%
\begin{equation}
H_{n}(\mathfrak{g},\mathfrak{g})\overset{I}{\longrightarrow}H_{n+1}%
(\mathfrak{g},k)\longrightarrow k\text{,} \label{lBT_revIIntro}%
\end{equation}
with $I$ a natural map to be explained in \S \ref{section_CEComplexes}. By the
Universal Coefficient Theorem for Lie algebras, $H^{n+1}(\mathfrak{g},k)\cong
H_{n+1}(\mathfrak{g},k)^{\ast}$, referring to the dual space. As such,
although $\phi$ is well-defined, the formula only applies to those cycles
admitting a lift under $I$ (as soon as it exists, the choice does not matter).
The formula is rather complicated. However, the pullback to particular
subalgebras of $\mathfrak{G}$ can be much nicer, for example for multiloop Lie
algebras of simple Lie algebras, we get the following:

\begin{theorem}
\label{intro_Thm3CocycleForLieAlgebras}Suppose $\mathfrak{g}/k$ is a
finite-dimensional centreless Lie algebra (e.g. simple). For $Y_{0}%
,\ldots,Y_{n}\in\mathfrak{g}$ we call%
\[
B(Y_{0},\ldots,Y_{n}):=\operatorname*{tr}\nolimits_{\operatorname*{End}%
_{k}(\mathfrak{g})}(\operatorname*{ad}(Y_{0})\operatorname*{ad}(Y_{1}%
)\cdots\operatorname*{ad}(Y_{n}))
\]
the `generalized Killing form'. Then on all Lie cycles admitting a lift under
$I$ as in eq. \ref{lBT_revIIntro}, the pullback of $\phi$ to $\mathfrak{g}%
[t_{1}^{\pm},\ldots,t_{n}^{\pm}]$ is explicitly given by%
\begin{align*}
&  \left.  \phi(Y_{0}t_{1}^{c_{0,1}}\cdots t_{n}^{c_{0,n}}\wedge\cdots\wedge
Y_{n}t_{1}^{c_{n,1}}\cdots t_{n}^{c_{n,n}})\right.  =\\
&  \qquad\left(  -1\right)  ^{n}\sum_{\pi\in\mathfrak{S}_{n}}%
\operatorname*{sgn}(\pi)B(Y_{\pi(1)},\ldots,Y_{\pi(n)},Y_{0})\prod
\limits_{i=1}^{n}c_{\pi(i),i}%
\end{align*}
whenever $\forall i\in\{1,\ldots,n\}:\sum_{p=0}^{n}c_{p,i}=0$ and zero
otherwise. Here $c_{i,p}\in\mathbf{Z}$ for all $i=0,\ldots,n$ and
$p=1,\ldots,n$.
\end{theorem}

If $\mathfrak{g}$ is finite-dimensional simple and $n=1$, then the class
$\phi$ yields the universal central extension of the loop Lie algebra
$\mathfrak{g}[t_{1},t_{1}^{-1}]$, the associated affine Lie algebra
$\widehat{\mathfrak{g}}$ (without extending by a derivation),%
\[
0\longrightarrow k\longrightarrow\widehat{\mathfrak{g}}\longrightarrow
\mathfrak{g}[t_{1},t_{1}^{-1}]\longrightarrow0\text{.}%
\]
In this case $B$ is obviously just the ordinary Killing form of $\mathfrak{g}
$. The above theorem will be proven in \S \ref{BT_sect_applications_multiloop}%
.\medskip

Additionally, we should say that these computations have an application
outside the theory of Lie algebras.\newline For this we need to return to the
roots of the subject. In 1967 J. Tate \cite{MR0227171} showed that the residue
of a rational $1$-form $f\mathrm{d}g$ at a closed point $x$ on an algebraic
curve $X/k$ can be expressed as a certain operator-theoretic trace on an
infinite-dimensional space. Arbarello, de Concini and Kac \cite[eq.
(2.7)]{MR1013132} reformulated this as
\begin{equation}
\operatorname*{res}\nolimits_{x}f\mathrm{d}g=\operatorname*{tr}([\pi
,g]f)\text{.} \label{lBTA_10}%
\end{equation}
On the right-hand side the functions $f,g$ are to be read as multiplication
operators acting on the local field $\operatorname*{Frac}\widehat{\mathcal{O}%
}_{X,x}\simeq\kappa(x)((t_{1}))$, seen as a $\kappa(x)$-vector space, and
$\pi$ denotes some projector on the non-principal part, e.g. \textquotedblleft
we cut off the principal part of the Laurent series\textquotedblright. It is
natural to ask whether there exists a generalization of this formula to higher
residues. We can give such a formula; it will be proven in
\S \ref{BT_sect_applications_residue}:

\begin{theorem}
\label{Prop_MainResidueThm}For a multiple Laurent polynomial ring with residue
field $k$, say%
\[
R:=k[t_{1}^{\pm},\ldots,t_{n}^{\pm}]\text{,}%
\]
and $f_{0},\ldots,f_{n}\in R$ we have%
\begin{align*}
&  \operatorname*{res}\nolimits_{t_{1}}\cdots\operatorname*{res}%
\nolimits_{t_{n}}f_{0}\mathrm{d}f_{1}\ldots\mathrm{d}f_{n}\\
&  =\left(  -1\right)  ^{n}\operatorname*{tr}%
{\textstyle\sum_{\pi\in\mathfrak{S}_{n}}}
\operatorname*{sgn}(\pi)%
{\textstyle\sum_{\gamma_{1}\ldots\gamma_{n}\in\{\pm\}}}
\left(  -1\right)  ^{\gamma_{1}+\cdots+\gamma_{n}}(P_{1}^{-\gamma_{1}%
}\operatorname*{ad}(f_{\pi(1)})P_{1}^{\gamma_{1}})\\
&  \qquad\qquad\qquad\cdots(P_{n}^{-\gamma_{n}}\operatorname*{ad}(f_{\pi
(n)})P_{n}^{\gamma_{n}})f_{0}\text{,}%
\end{align*}
where $P_{1}^{\pm},\ldots,P_{n}^{\pm}$ are suitable projectors (explained in
\S \ref{BT_sect_applications_residue}, eq. \ref{lBTA_32}).

\begin{enumerate}
\item \label{resthm_part1}For $n=1$ and $\pi:=P_{1}^{+}$ the formula reduces
to the familiar eq. \ref{lBTA_10}, as in \cite{MR1013132}.

\item \label{resthm_part2}If we have $f_{i}=t_{1}^{c_{i,1}}\cdots
t_{n}^{c_{i,n}}$ for $i=0,\ldots,n$, the formula reduces to%
\[
\operatorname*{res}f_{0}\mathrm{d}f_{1}\ldots\mathrm{d}f_{n}=\det%
\begin{pmatrix}
c_{1,1} & \cdots & c_{n,1}\\
\vdots & \ddots & \vdots\\
c_{1,n} & \cdots & c_{n,n}%
\end{pmatrix}
\qquad\text{if }\forall i:\sum_{p=0}^{n}c_{p,i}=0
\]
and the residue is zero if the condition on the right-hand side is not satisfied.

\item \label{resthm_part3}For $n=1$ and $f_{1}=t_{1}$ this reduces by
linearity to the classical definition%
\[
\operatorname*{res}\alpha t_{1}^{c_{1}}\mathrm{d}t_{1}=\left\{
\begin{array}
[c]{cl}%
\alpha & \text{if }c_{1}=-1\\
0 & \text{if }c_{1}\neq-1\text{.}%
\end{array}
\right.
\]

\end{enumerate}
\end{theorem}

How to construct the cocycle $\phi$?\medskip

There are various ways to approach this construction. Frenkel and Zhu
\cite{Frenkel08092011} use distinguished generators of the cohomology ring of
infinite matrix algebras, based on computations of Feigin and Tsygan
\cite{MR705056}. This is a very natural approach. However, in this text we use
a different approach based on Beilinson's multidimensional adelic residue
\cite{MR565095}. Originally, this approach was only used to generalize Tate's
approach to the residue symbol to several variables, but it readily
generalizes to the problem we are discussing here. This might be interesting
also since \cite{MR565095} does not give an explicit formula -- and it is not
totally trivial to extrapolate a formula from the definition:

\begin{theorem}
\label{intro_Thm5CocycleAgreesWithBeilinsons}The formula in Thm.
\ref{Prop_MainResidueThm} arises from the construction of Beilinson in the
paper \cite[Lemma 1]{MR565095}, i.e. it is the composition%
\begin{align}
&  \Omega_{R/k}^{n}\overset{\left(  -1\right)  ^{n}\varkappa}{\longrightarrow
}H_{n+1}^{\operatorname*{Lie}}(\mathfrak{G},k)\overset{\rho_{2}}%
{\longrightarrow}\left.  ^{\wedge}E_{0,n+1}^{n+1}\right. \label{lBTAA_3}\\
&  \qquad\qquad\overset{(d_{n+1})^{-1}}{\longrightarrow}\left.  ^{\wedge
}E_{n+1,1}^{n+1}\right.  \overset{\rho_{1}}{\longrightarrow}H_{0}%
^{\operatorname*{Lie}}(\mathfrak{G},N^{n+1})\overset{\operatorname*{tr}%
}{\longrightarrow}k\text{,}\nonumber
\end{align}
where

\begin{itemize}
\item $\varkappa:f_{0}\mathrm{d}f_{1}\wedge\cdots\wedge\mathrm{d}f_{n}\mapsto
f_{0}\wedge\cdots\wedge f_{n}$,

\item $N^{n+1}$ is a certain $\mathfrak{G}$-module (see
\S \ref{section_CubeComplex} for the definition, or $T_{\ast N}$ in
\cite{MR565095}), and

\item $\rho_{1},\rho_{2}$ are edge maps and $d_{n+1}$ a differential on the
$(n+1)^{\text{th}}$ page of a certain spectral sequence $\left.  ^{\wedge
}E_{\bullet,\bullet}^{\bullet}\right.  $ (constructed in Lemma
\ref{BT_PropExplicitDifferentialInSpecSeq}, or see \cite[Lemma 1]%
{MR565095}).\smallskip
\end{itemize}
\end{theorem}

This result is only meaningful to readers familiar with the paper
\cite{MR565095}.\medskip

The above theorem actually lies at the heart of our approach. We formulate a
contracting homotopy for a mild variation of the relevant complexes in
\cite{MR565095} and then, in a slightly tedious computation, make the spectral
sequence differential $d_{n+1}$ explicit on the basis of this.\medskip

Finally, for applications in algebraic geometry, e.g. the interpretation as a
local residue, it is unfortunate to interpret the word \textquotedblleft loop
Lie algebra\textquotedblright\ as $\mathfrak{g}[t,t^{-1}]$. It is better to
work with Laurent series, i.e. $\mathfrak{g}((t))$, or even local components
of ad\`{e}les. Tate's original work uses the language of ad\`{e}les for
example. For this reason, we shall axiomatize all these variations through the
notion of a \textquotedblleft cubically decomposed algebra\textquotedblright%
\ (essentially taken from \cite{MR565095}, where it's not given a name).

\subsection{Acknowledgements}

I am very thankful to Ivan Fesenko and Matthew Morrow for many valuable
discussions, especially on an ad\`{e}le interpretation. I thank the Research
Group of Prof. Marc Levine for the stimulating scientific environment. I
heartily thank the anonymous referee for greatly improving the presentation,
especially in \S \ref{section_CubeComplex}, and observing the fact $H^{2}=0$
in eq. \ref{lBTA_35}, which clarifies a crucial cancellation in the proof of
Prop. \ref{BT_KeyLemmaFormula}.

\subsection{What is not here}

In the present text I only discuss the `linear algebra setting' of Tate's
central extension (\cite[\S 1]{BFM_Conformal} for the case $n=1$). There is
also a `differential operator setting' (\cite[\S 2]{BFM_Conformal}), which I
will treat in a future text. Roughly speaking, $\mathfrak{G}$ will be replaced
by much smaller algebras of differential operators on a vector bundle.

Moreover, I do not treat the true multiloop analogue of an affine Kac-Moody
algebra in the present text. Already for $n=1$ I only consider the `plain'
affine Lie algebras without extending by a derivation. From the perspective of
a triangular decomposition, this is a rather horrible omission: the root
spaces are infinite-dimensional! However, as the reader can probably imagine
from the computations in \S \ref{BT_sect_applications_residue},
\S \ref{BT_sect_applications_multiloop} the calculation gets a lot more
complicated in the presence of derivations. Thus, this aspect will also be
deferred to a future text. The same applies to the analogue of the plain
Virasoro algebra. There should also be a nonlinear analogue, distinguished
cohomology classes for multiloop groups. The cases $n=1,2$ (along with a
higher representation theory in categories) are treated in detail by Frenkel
and Zhu in \cite{Frenkel08092011}.

One should also mention that there are completely orthogonal generalizations
of Kac-Moody/Virasoro cocycles to multiloop Lie algebras, see for example
\cite[\S 9]{MR934284}, \cite{MR2743761}.

\section{\label{section_Frameworks}Basic framework}

For an associative algebra $A$ we shall write $A_{Lie}$ to denote the
associated Lie algebra.

\begin{definition}
[\cite{MR565095}]\label{BT_DefCubicallyDecompAlgebra}An \emph{(}%
$n$\emph{-fold)} \emph{cubically decomposed algebra} (over a field $k$) is the
datum $(A,(I_{i}^{\pm}),\tau)$:

\begin{itemize}
\item an associative unital (not necessarily commutative) $k$-algebra $A$;

\item two-sided ideals $I_{i}^{+},I_{i}^{-}$ such that $I_{i}^{+}+I_{i}^{-}=A
$ for $i=1,\ldots,n$;

\item writing $I_{i}^{0}:=I_{i}^{+}\cap I_{i}^{-}$ and $I_{\operatorname*{tr}%
}:=I_{1}^{0}\cap\cdots\cap I_{n}^{0}$, a $k$-linear map%
\[
\tau:I_{\operatorname*{tr},Lie}/[I_{\operatorname*{tr},Lie},A_{Lie}%
]\rightarrow k\text{.}%
\]

\end{itemize}
\end{definition}

For any finite-dimensional $k$-vector space $V$ certain infinite matrix
algebras act naturally on the $k$-vector space of multiple Laurent polynomials
$V[t_{1}^{\pm1},\ldots,t_{n}^{\pm1}]$. This yields an example of this
structure, see \S \ref{TATE_section_InfiniteMatrixAlgebras} below. There is
also an analogue for $V((t_{1}))\cdots((t_{n}))$, which we leave to the reader
to formulate (this links to higher local fields, see \cite{MR1804915}). Local
components of Parshin-Beilinson ad\`{e}les of schemes yield another example,
see \cite[\S 1]{MR565095}. In \textit{loc. cit.} the ideals $I_{i}^{+}%
,I_{i}^{-}$ are called $X^{i},Y^{i}$. The latter gives the multidimensional
generalization of the ad\`{e}le formulation of Tate \cite{MR0227171}. See
\cite{MR2658047}, \cite{0763.14006}, \cite{MR1374916},
\cite{MorrowResiduePaper} for more background on higher-dimensional ad\`{e}les
and their uses.

\subsection{\label{TATE_section_InfiniteMatrixAlgebras}Infinite matrix
algebras}

Fix a field $k$. Let $R$ be an associative $k$-algebra, not necessarily unital
or commutative. Define an algebra of infinite matrices%
\begin{equation}
E(R):=\{\phi=(\phi_{ij})_{i,j\in\mathbf{Z}},\phi_{ij}\in R\mid\exists K_{\phi
}:\left\vert i-j\right\vert >K_{\phi}\Rightarrow\phi_{ij}=0\}\text{.}
\label{TATEMATRIX_l1}%
\end{equation}
Define a product by $(\phi\cdot\phi^{\prime})_{ik}:=\sum_{j\in\mathbf{Z}}%
\phi_{ij}\phi_{jk}^{\prime}$, the usual matrix multiplication formula; this
sum only has finitely many non-zero terms and one can choose $K_{\phi
\phi^{\prime}}:=K_{\phi}+K_{\phi^{\prime}}$. Then $E(R)$ becomes an
associative $k$-algebra. If $R$ is unital, $E(R)$ is also unital. $E$ is a
functor from associative algebras to associative algebras; for a morphism
$\varphi:R\rightarrow S$ there is an induced morphism $E(\varphi
):E(R)\rightarrow E(S)$ by using $\varphi$ entry-by-entry, i.e. $(E(\varphi
)\phi)_{ij}:=\varphi(\phi_{ij})$. If $I\subseteq R$ is an ideal (which is in
particular a non-unital associative ring), $E(I)\subseteq E(R)$ is an ideal.
Moreover, for ideals $I_{1},I_{2}$ one has $E(I_{1}\cap I_{2})=E(I_{1})\cap
E(I_{2})$ and $E(I_{1}+I_{2})=E(I_{1})+E(I_{2})$, as a sum of ideals. Next,
define%
\begin{align*}
I^{+}(R):=  &  \{\phi\in E(R)\mid\exists B_{\phi}:i<B_{\phi}\Rightarrow
\phi_{ij}=0\}\\
I^{-}(R):=  &  \{\phi\in E(R)\mid\exists B_{\phi}:j>B_{\phi}\Rightarrow
\phi_{ij}=0\}
\end{align*}
and one checks easily that $I^{+}(R),I^{-}(R)$ are two-sided ideals in $E(R)$.
The following figure attempts to visualize the shape of the matrices in
$E(R),I^{+}(R)$ and $I^{-}(R)$ respectively:%
\begin{center}
\includegraphics[
height=0.889in,
width=2.7994in
]%
{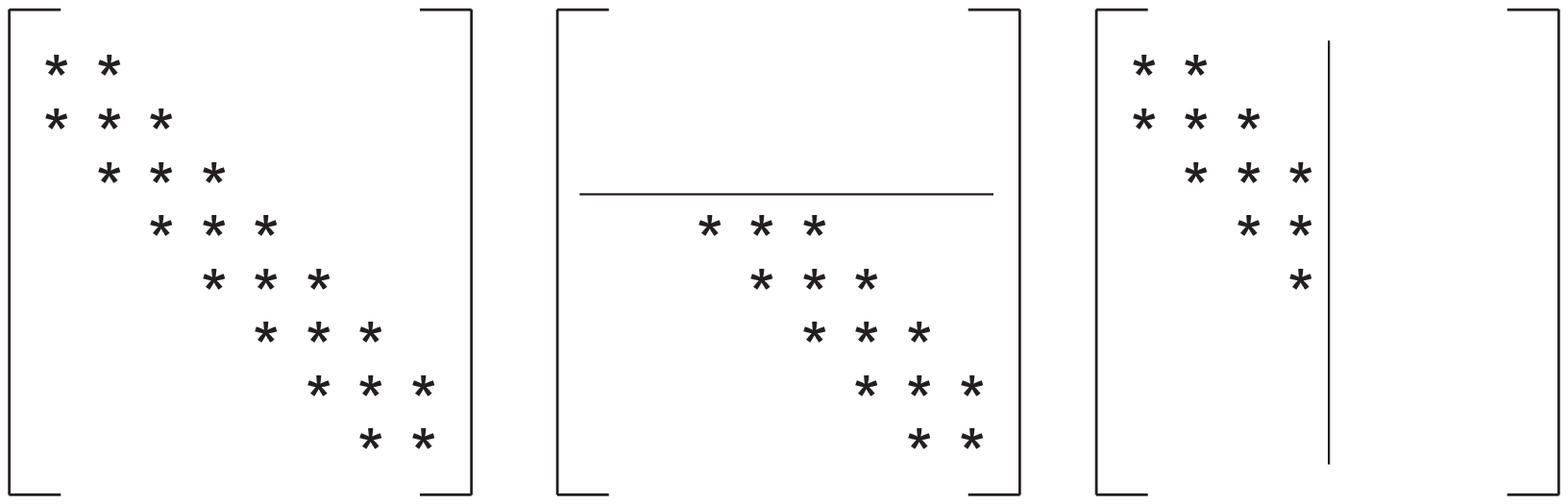}%
\end{center}
Define $I^{0}(R):=I^{+}(R)\cap I^{-}(R)$ and one checks that%
\[
I^{0}(R):=\{\phi\in E(R)\mid\phi_{ij}=0\text{ for all but finitely many
}\left(  i,j\right)  \}\text{.}%
\]
There is a trace morphism%
\begin{equation}
\operatorname*{tr}:I^{0}(R)\rightarrow R\text{;}\qquad\operatorname*{tr}\phi:=%
{\textstyle\sum\nolimits_{i\in\mathbf{Z}}}
\phi_{ii}\text{,} \label{TATEMATRIX_l3}%
\end{equation}
the sum is obviously finite. One easily verifies that $\operatorname*{tr}%
[\phi,\phi^{\prime}]=%
{\textstyle\sum\nolimits_{i,j\in\mathbf{Z}}}
[\phi_{ij},\phi_{ji}^{\prime}]$ and thus $\operatorname*{tr}[I^{0}%
(R),E(R)]\subseteq\lbrack R,R] $. More generally, if $R^{\prime}\subseteq R$
is a subalgebra,%
\[
\operatorname*{tr}[I^{0}(R^{\prime}),E(R)]\subseteq\lbrack R^{\prime
},R]\text{.}%
\]
We note that this trace does not necessarily vanish on commutators. Moreover,
every $\phi\in E(R)$ can be written as $\phi=\phi^{+}+\phi^{-}$ with
$\phi_{ij}^{+}:=\delta_{i\geq0}\phi_{ij}$ (for this $R$ need not be unital,
use $\phi_{ij}$ for $i\geq0$ and $0$ otherwise) and $\phi^{-}=\phi-\phi^{+}$.
One checks that $\phi^{\pm}\in I^{\pm}(R)$. It follows that $I^{+}%
(R)+I^{-}(R)=E(R)$.

Finally, let $M$ be an $R$-bimodule (over $k$, i.e. a left-$(A\otimes
_{k}A^{op})$-module; $R$-bimodules form an abelian category). Analogously to
$E(R)$, define%
\begin{equation}
E(M):=\{\phi=(\phi_{ij})_{i,j\in\mathbf{Z}},\phi_{ij}\in M\mid\exists K_{\phi
}:\left\vert i-j\right\vert >K_{\phi}\Rightarrow\phi_{ij}=0\}\text{.}
\label{TATEMATRIX_l2}%
\end{equation}
Again using the matrix multiplication formula, $E(M)$ is an $E(R)$-bimodule.
If $0\rightarrow M^{\prime}\rightarrow M\rightarrow M^{\prime\prime
}\rightarrow0$ is an exact sequence of $R$-bimodules, $0\rightarrow
E(M^{\prime})\rightarrow E(M)\rightarrow E(M^{\prime\prime})\rightarrow0$ is
an exact sequence of $E(R)$-bimodules. Note that for an ideal $I\subseteq R$
the object $E(I)$ is well-defined, regardless whether we regard $I$ as an
associative ring as in eq. \ref{TATEMATRIX_l1} or an $R$-bimodule as in eq.
\ref{TATEMATRIX_l2}.

Now let $V$ be a finite-dimensional $k$-vector space and $R_{0}$ an arbitrary
unital subalgebra of $\operatorname*{End}_{k}(V)$. Define $R_{i}:=E(R_{i-1})$
for $i=1,\ldots,n$. Note that via $k\rightarrow R_{0}$, $\alpha\mapsto
\alpha\cdot\mathbf{1}_{\operatorname*{End}_{k}(V)}$, $k$ is embedded into the
center of $R_{i}$. Then $R_{n}=(E\circ\cdots\circ E)(R_{0})$ is a unital
associative $k$-algebra. Its elements may be indexed $\phi=(\phi_{(i_{n}%
,j_{n}),\ldots,(i_{1},j_{1})\in\mathbf{Z}^{2n}}\in R_{0})$. By the properties
discussed above,%
\[
I_{i}^{\pm}:=(\underset{n}{E}\cdots\underset{i+1}{E}\circ\underset{i}{I^{\pm}%
}\circ\underset{i-1}{E}\cdots\underset{1}{E})(R_{0})\qquad\text{(}I^{\pm
}\text{ in the }i\text{-th place),}%
\]
is an ideal in $R_{n}$ (we use centered subscripts only to emphasize the
numbering). Moreover,%
\begin{align*}
I_{i}^{+}+I_{i}^{-}  &  =(E\cdots E\circ I^{+}\circ E\cdots E)(R_{0})+(E\cdots
E\circ I^{-}\circ E\cdots E)(R_{0})\\
&  =(E\cdots E\circ E\circ E\cdots E)(R_{0})=R_{n}\text{.}%
\end{align*}
By composing the traces of eq. \ref{TATEMATRIX_l3} we arrive at a $k$-linear
map $\tau$,%
\begin{align*}
\tau &  :I_{\operatorname*{tr}}=I_{1}^{0}\cap\cdots\cap I_{n}^{0}=(I^{0}%
\circ\cdots\circ I^{0})(R_{0})\\
&  \qquad\overset{\operatorname*{tr}}{\longrightarrow}\cdots\overset
{\operatorname*{tr}}{\longrightarrow}I^{0}(I^{0}(R_{0}))\overset
{\operatorname*{tr}}{\longrightarrow}I^{0}(R_{0})\overset{\operatorname*{tr}%
}{\longrightarrow}R_{0}\overset{\operatorname*{Tr}}{\longrightarrow}k\text{,}%
\end{align*}
where \textquotedblleft$\operatorname*{Tr}$\textquotedblright\ (as opposed to
\textquotedblleft$\operatorname*{tr}$\textquotedblright) denotes the ordinary
matrix trace of $\operatorname*{End}_{k}(V)$ ($\supseteq R_{0}$). Here we have
used that $V$ is finite-dimensional over $k$. Using $\operatorname*{tr}%
[I^{0}(R^{\prime}),E(R)]\subseteq\lbrack R^{\prime},R]$ (for subalgebras
$R^{\prime}\subseteq R$) inductively, one sees that%
\begin{align*}
\tau\lbrack I_{\operatorname*{tr}},R_{n}]  &  =\operatorname*{Tr}%
(\operatorname*{tr}\circ\cdots\circ\operatorname*{tr}\circ\operatorname*{tr}%
)[I^{0}(I^{0}(\cdots)),E(E(\cdots))]\\
&  \subseteq\operatorname*{Tr}(\operatorname*{tr}\circ\cdots\circ
\operatorname*{tr})[I^{0}(\cdots),E(\cdots)]\subseteq\operatorname*{Tr}%
[R_{0},R_{0}]=0
\end{align*}
since the ordinary trace $\operatorname*{Tr}$ vanishes on commutators. Hence,
$\tau$ factors to a morphism $\tau:I_{\operatorname*{tr},Lie}%
/[I_{\operatorname*{tr},Lie},R_{Lie}]\rightarrow k$. Summarizing, for every
$n\geq1$, every finite-dimensional $k$-vector space $V$ and every unital
subalgebra $R_{0}\subseteq\operatorname*{End}_{k}(V)$, $(R_{n},(I_{i}^{\pm
}),\tau)$ is a cubically decomposed algebra.\newline Finally, note that for
any associative algebra $R$, $E(R)$ is a right-$R$-submodule of \textit{right}%
-$R$-module endomorphisms $\operatorname*{End}_{R}(R[t,t^{-1}])$ of
$R[t,t^{-1}]$. Write elements as $a=\sum_{i\in\mathbf{Z}}a_{i}t^{i}$, also
denoted $a=(a_{i})_{i}$ with $a_{i}\in R$, and let $\phi=(\phi_{ij})$ act by
$\left(  \phi\cdot a\right)  _{i}:=\sum_{k}\phi_{ik}a_{k}$. Moreover, each
$a\in R[t,t^{-1}]$ determines a right-$R$-module endomorphism via the
multiplication operator $x\mapsto a\cdot x$. We find%
\[
R[t,t^{-1}]\hookrightarrow E(R)\hookrightarrow\operatorname*{End}%
\nolimits_{R}(R[t,t^{-1}])\text{.}%
\]
Multiplication with $t^{i}$ is represented by a matrix with a diagonal
$\ldots,1,1,1,\ldots$, shifted by $i$ off the principal diagonal. Inductively,%
\begin{equation}
R_{0}[t_{1}^{\pm1},\ldots,t_{n}^{\pm1}]\hookrightarrow R_{n}\hookrightarrow
\operatorname*{End}\nolimits_{R_{0}}(R_{0}[t_{1}^{\pm1},\ldots,t_{n}^{\pm
1}])\text{.} \label{lBTA_30}%
\end{equation}
See for example \cite[\S 1]{MR723457}, \cite[Lec. 4]{MR1021978} for more
information regarding the case $n=1$ and \cite[\S 3]{Frenkel08092011} for a
similar procedure when $n=2$.

\section{\label{section_CEComplexes}Modified Chevalley-Eilenberg complexes}

Suppose $k$ is a field and $\mathfrak{g}$ a Lie algebra over $k$. We recall
that for any $\mathfrak{g}$-module the conventional Chevalley-Eilenberg
complex is given by $C(M)_{r}:=M\otimes%
{\textstyle\bigwedge\nolimits^{r}}
\mathfrak{g}$ along with the differential%
\begin{align}
\delta &  :=\delta^{\lbrack1]}+\delta^{\lbrack2]}:C(M)_{r}\rightarrow
C(M)_{r-1}\label{lBT_CEComplexDifferential}\\
\delta^{\lbrack1]}(f_{0}\otimes f_{1}\wedge\ldots\wedge f_{r})  &  :=%
{\textstyle\sum\nolimits_{i=1}^{r}}
(-1)^{i}[f_{0},f_{i}]\otimes f_{1}\wedge\ldots\wedge\widehat{f_{i}}%
\wedge\ldots\wedge f_{r}\nonumber\\
\delta^{\lbrack2]}(f_{0}\otimes f_{1}\wedge\ldots\wedge f_{r})  &  :=%
{\textstyle\sum\nolimits_{1\leq i<j\leq r}}
(-1)^{i+j+1}f_{0}\otimes\lbrack f_{i},f_{j}]\wedge f_{1}\ldots\widehat{f_{i}%
}\ldots\widehat{f_{j}}\ldots\wedge f_{r}\nonumber
\end{align}
for $f_{0}\in M$ and $f_{1},\ldots,f_{r}\in\mathfrak{g}$. Its homology is (if
one wants by definition) Lie homology with coefficients in $M$. There is also
a cohomological analogue; we refer the reader to the literature for details,
e.g. \cite[Ch. 10]{MR1217970}. We may view $k$ itself as a $\mathfrak{g}%
$-module with the trivial structure. There is an obvious morphism%
\begin{equation}
I:C(\mathfrak{g})_{r}\rightarrow C(k)_{r+1}\qquad f_{0}\otimes f_{1}%
\wedge\ldots\wedge f_{r}\mapsto\left(  -1\right)  ^{r}\mathbf{1}_{k}\otimes
f_{0}\wedge f_{1}\wedge\ldots\wedge f_{r} \label{lBT_revLieImap}%
\end{equation}
and one checks easily that this commutes with the respective differentials and
thus induces morphisms $H_{r}\left(  \mathfrak{g},\mathfrak{g}\right)
\rightarrow H_{r+1}\left(  \mathfrak{g},k\right)  $. The linear dual
$\mathfrak{g}^{\ast}:=\operatorname*{Hom}\nolimits_{k}(\mathfrak{g},k)$ is
canonically a $\mathfrak{g}$-module via $\left(  f\cdot\varphi\right)
(g):=\varphi([g,f])$ for $\varphi\in\mathfrak{g}^{\ast}$ and $f,g\in
\mathfrak{g}$. The cohomological analogue of eq. \ref{lBT_revLieImap} is the
morphism $I:H^{r+1}\left(  \mathfrak{g},k\right)  \rightarrow H^{r}\left(
\mathfrak{g},\mathfrak{g}^{\ast}\right)  $ given by%
\[
(I\phi)(f_{1}\wedge\ldots\wedge f_{r})(f_{0}):=\left(  -1\right)  ^{r}%
\phi(f_{0}\wedge f_{1}\wedge\ldots\wedge f_{r})\text{.}%
\]

\begin{remark}
\label{lBT_revCyclicAnalogue}These maps could be viewed as a Lie-theoretic
analogue of map $I$ in Connes' periodicity sequence, see \cite[\S 2.2]%
{MR1217970}. We may view $H_{\ast-1}(\mathfrak{g},\mathfrak{g})$ as a partial
\textquotedblleft uncyclic\textquotedblright\ counterpart of Lie homology. The
true Hochschild analogue would be Leibniz homology, cf. \cite[\S 10.6]%
{MR1217970}. For the present purposes we have however no use for this analogue.
\end{remark}

Let $\mathfrak{j}\subseteq\mathfrak{g}$ be a Lie ideal. As such, it is a
$\mathfrak{g}$-module and we may consider $C(\mathfrak{j})_{\bullet}$.
Following \cite{MR565095} we may work with a `cyclically symmetrized'
counterpart: We write $\mathfrak{j}\wedge%
{\textstyle\bigwedge\nolimits^{r-1}}
\mathfrak{g}$ to denote the $\mathfrak{g}$-submodule of $\mathfrak{g}\wedge%
{\textstyle\bigwedge\nolimits^{r-1}}
\mathfrak{g}=%
{\textstyle\bigwedge\nolimits^{r}}
\mathfrak{g}$ generated by elements $j\wedge f_{1}\wedge\ldots\wedge f_{r-1}$
such that $j\in\mathfrak{j}$ and $f_{1},\ldots,f_{r-1}\in\mathfrak{g}$. If
$\mathfrak{j}_{i}$, $i=1,2,\ldots$, are Lie ideals, we denote by
$(\bigoplus_{i}\mathfrak{j}_{i})\wedge%
{\textstyle\bigwedge\nolimits^{r-1}}
\mathfrak{g}$ the module $\bigoplus_{i}(\mathfrak{j}_{i}\wedge%
{\textstyle\bigwedge\nolimits^{r-1}}
\mathfrak{g})$.

\begin{example}
If $k\left\langle s,t,u\right\rangle \,$and $k\left\langle s\right\rangle $
denote a $3$-dimensional abelian Lie algebra along with a $1$-dimensional Lie
ideal, then $%
{\textstyle\bigwedge\nolimits^{2}}
k\left\langle s,t,u\right\rangle $ is $3$-dimensional with basis $s\wedge t$,
$s\wedge u$ and $t\wedge u$. Then $k\left\langle s\right\rangle \wedge
k\left\langle s,t,u\right\rangle $ is $2$-dimensional with basis $s\wedge t$,
$s\wedge u$.
\end{example}

The $k$-vector spaces $CE(\mathfrak{j})_{r}:=\mathfrak{j}\wedge%
{\textstyle\bigwedge\nolimits^{r-1}}
\mathfrak{g}$ (for $r\geq1$) and $CE(\mathfrak{j})_{0}:=k$ define a subcomplex
of $C(k)_{\bullet}$. In particular, the differential is given by%
\begin{equation}
\delta(f_{0}\wedge f_{1}\wedge\ldots\wedge f_{r}):=%
{\textstyle\sum\nolimits_{0\leq i<j\leq r}}
(-1)^{i+j}[f_{i},f_{j}]\wedge f_{0}\wedge\ldots\widehat{f_{i}}\ldots
\widehat{f_{j}}\ldots\wedge f_{r}\text{.} \label{lBT_effectiveCEdifferential}%
\end{equation}
It is well-defined since $\mathfrak{j}$ is a Lie ideal. We get morphisms
generalizing $I$, notably $H_{r}(\mathfrak{g},\mathfrak{j})\rightarrow
H_{r+1}(CE(\mathfrak{j}))$ via $\mathfrak{j}\otimes%
{\textstyle\bigwedge\nolimits^{r}}
\mathfrak{g}\rightarrow\mathfrak{j}\wedge%
{\textstyle\bigwedge\nolimits^{r}}
\mathfrak{g}$ and analogously $H^{r+1}(CE(\mathfrak{j}))\rightarrow
H^{r}(\mathfrak{g},\mathfrak{j}^{\ast})$. We have resisted the temptation to
re-index $CE(-)_{\bullet}$ despite the unpleasant $\left(  +1\right)  $-shift
in eq. \ref{lBT_revLieImap} in order to remain compatible with standard usage
in the following sense:

\begin{lemma}
[{\cite[Lemma 1(a)]{MR565095}}]\label{BT_LemmaOnComputingLieHomology}%
$CE(\mathfrak{g})_{\bullet}$ is a complex of $k$-vector spaces and is
quasi-isomorphic to $k\otimes_{U\mathfrak{g}}^{\mathbf{L}}k$. In particular%
\[
H_{i}(\mathfrak{g},k)=H_{i}(CE(\mathfrak{g})_{\bullet})\text{ and }%
H^{i}(\mathfrak{g},k)=H^{i}(\operatorname*{Hom}\nolimits_{k}(CE(\mathfrak{g}%
)_{\bullet},k))\text{.}%
\]

\end{lemma}

\begin{proof}
As we have explained above, $CE(\mathfrak{g})_{\bullet}$ agrees with the
standard Chevalley-Eilenberg complex and the latter is well-known to represent
$k\otimes_{U\mathfrak{g}}^{\mathbf{L}}k$.
\end{proof}

We easily compute%
\begin{align}
&  H_{0}(\mathfrak{g},\mathfrak{j})\overset{\cong}{\underset{I}%
{\longrightarrow}}H_{1}(CE(\mathfrak{j}))\cong\mathfrak{j}/[\mathfrak{g}%
,\mathfrak{j}]\label{lBTrev_curious_identities}\\
&  H^{1}(CE(\mathfrak{j}))\overset{\cong}{\underset{I}{\longrightarrow}}%
H^{0}(\mathfrak{g},\mathfrak{j}^{\ast})\cong\left(  \mathfrak{j}%
/[\mathfrak{g},\mathfrak{j}]\right)  ^{\ast}\text{.}\nonumber
\end{align}
In higher degrees the map $I$ ceases to be an isomorphism.

Nonetheless, this computation hints at the principle of computation which we
shall use below. Beilinson uses $CE(-)_{\bullet}$ in his paper \cite{MR565095}%
, whereas we will only be able to do manageable computations with
$C(-)_{\bullet}$. The map $I$ will serve to deduce facts about $CE(-)_{\bullet
}$ while working with $C(-)_{\bullet}$.

\section{Cubically decomposed algebras}

Let $(A,(I_{i}^{\pm}),\tau)$ be an $n$-fold cubically decomposed algebra over
a field $k$, see Def. \ref{BT_DefCubicallyDecompAlgebra}, i.e. we are given
the following datum:

\begin{itemize}
\item an associative unital (not necessarily commutative) $k$-algebra $A$;

\item two-sided ideals $I_{i}^{+},I_{i}^{-}$ such that $I_{i}^{+}+I_{i}^{-}=A
$ for $i=1,\ldots,n$;

\item writing $I_{i}^{0}:=I_{i}^{+}\cap I_{i}^{-}$ and $I_{\operatorname*{tr}%
}:=I_{1}^{0}\cap\cdots\cap I_{n}^{0}$, a $k$-linear map%
\[
\tau:I_{\operatorname*{tr},Lie}/[I_{\operatorname*{tr},Lie},A_{Lie}%
]\rightarrow k\text{.}%
\]

\end{itemize}

See \S \ref{section_Frameworks} to see how this type of structure arises. As a
shorthand, define $\mathfrak{g}:=A_{Lie}$. For any elements $s_{1}%
,\ldots,s_{n}\in\{+,-,0\}$ we define the \emph{degree }$\deg(s_{1}%
,\ldots,s_{n}):=1+\#\{i\mid s_{i}=0\}$. Next, following \cite{MR565095} we
shall construct complexes of $\mathfrak{g}$-modules:

\begin{definition}
[\cite{MR565095}]For every $1\leq p\leq n+1$ define%
\begin{equation}
\left.  ^{\wedge}T_{\bullet}^{p}\right.  :=\coprod_{\substack{s_{1}%
,\ldots,s_{n}\in\{\pm,0\}\\\deg(s_{1}\ldots s_{n})=p}}\bigcap_{i=1}%
^{n}\left\{
\begin{array}
[c]{ll}%
CE(I_{i}^{+})_{\bullet} & \text{for }s_{i}=+\\
CE(I_{i}^{-})_{\bullet} & \text{for }s_{i}=-\\
CE(I_{i}^{+})_{\bullet}\cap CE(I_{i}^{-})_{\bullet} & \text{for }s_{i}=0
\end{array}
\right.  \label{lBT_DefComplexTWedge}%
\end{equation}
and $\left.  ^{\wedge}T_{\bullet}^{0}\right.  :=CE(\mathfrak{g})_{\bullet}$.
\end{definition}

Each $CE(I_{i}^{\pm})_{\bullet}$ is a complex and all their differentials are
defined by the same formula, eq. \ref{lBT_effectiveCEdifferential}, as such
the intersection of these complexes has a well-defined differential and is a
complex itself. Same for the coproduct. The complex $\left.  ^{\wedge
}T_{\bullet}^{\bullet}\right.  $ is inspired by a cubical object used by
Beilinson \cite{MR565095}.

\begin{example}
For $n=2$ we get complexes%
\begin{align*}
\left.  ^{\wedge}T_{\bullet}^{1}\right.   &  =%
{\textstyle\coprod\nolimits_{s_{1},s_{2}\in\{\pm\}}}
CE(I_{1}^{s_{1}})_{\bullet}\cap CE(I_{2}^{s_{2}})_{\bullet}\\
\left.  ^{\wedge}T_{\bullet}^{2}\right.   &  =%
{\textstyle\coprod\nolimits_{s_{1}\in\{\pm\}}}
CE(I_{1}^{s_{1}})_{\bullet}\cap CE(I_{2}^{+})_{\bullet}\cap CE(I_{2}%
^{-})_{\bullet}\\
&  \qquad\oplus%
{\textstyle\coprod\nolimits_{s_{2}\in\{\pm\}}}
CE(I_{1}^{+})_{\bullet}\cap CE(I_{1}^{-})_{\bullet}\cap CE(I_{2}^{s_{2}%
})_{\bullet}\\
\left.  ^{\wedge}T_{\bullet}^{3}\right.   &  =CE(I_{1}^{+})_{\bullet}\cap
CE(I_{1}^{-})_{\bullet}\cap CE(I_{2}^{+})_{\bullet}\cap CE(I_{2}^{-}%
)_{\bullet}\text{.}%
\end{align*}
Note that $CE(I_{1}^{+})_{\bullet}\cap CE(I_{1}^{-})_{\bullet}\neq
CE(I_{1}^{+}\cap I_{1}^{-})_{\bullet}$, e.g. $I_{1}^{+}\wedge I_{1}^{-}$ is a
subspace in degree two of the left-hand side, but not of the right-hand side.
\end{example}

Diverging from \cite{MR565095} we shall primarily use the following slightly
different auxiliary construction (which we will later relate to the above one):

\begin{definition}
For $1\leq p\leq n+1$ let%
\begin{equation}
\left.  ^{\otimes}T_{\bullet}^{p}\right.  :=\coprod_{\substack{s_{1}%
,\ldots,s_{n}\in\{\pm,0\}\\\deg(s_{1}\ldots s_{n})=p}}C(I_{1}^{s_{1}}\cap
I_{2}^{s_{2}}\cap\cdots\cap I_{n}^{s_{n}})_{\bullet}
\label{lBT_DefComplexTTensor}%
\end{equation}
and $\left.  ^{\otimes}T_{\bullet}^{p}\right.  :=C(\mathfrak{g})_{\bullet}$.
\end{definition}

So, instead of the modified Chevalley-Eilenberg complex of
\S \ref{section_CEComplexes} we just use the standard complexes for Lie
homology with suitable coefficients. Clearly the morphism $I:C(\mathfrak{g}%
)_{r}\rightarrow C(k)_{r+1}$ descends to morphisms%
\begin{align*}
C(\mathfrak{g})_{r}\supseteq\qquad C(I_{i}^{s_{i}})_{r}  &  \rightarrow
CE(I_{i}^{s_{i}})_{r+1}\qquad\subseteq C(k)_{r+1}\\
\underset{\in I_{i}^{s_{i}}}{f_{0}}\otimes f_{1}\wedge\ldots\wedge f_{r}  &
\mapsto\left(  -1\right)  ^{r}\underset{\in I_{i}^{s_{i}}}{f_{0}}\wedge
f_{1}\wedge\ldots\wedge f_{r}%
\end{align*}
As we take intersections of Lie ideals on the left $C(I_{1}^{s_{1}}\cap
\ldots)_{\bullet}$, as in eq. \ref{lBT_DefComplexTTensor}, the image lies in
the intersection of the individual images, i.e. $CE(I_{1}^{s_{i}})_{\bullet
}\cap\ldots$, as in eq. \ref{lBT_DefComplexTWedge}. As a result, we obtain
morphisms%
\[
\left.  ^{\otimes}T_{\bullet}^{p}\right.  \overset{I}{\longrightarrow}\left.
^{\wedge}T_{\bullet+1}^{p}\right.  \qquad\text{(for all }p\text{)}%
\]
and since they are a restriction of the map $I$ to subcomplexes, this is a
morphism of complexes, and thus induces maps on homology.

\section{\label{section_CubeComplex}The cube complex}

Next, we shall define maps $\cdots\rightarrow\left.  ^{\otimes}T_{\bullet}%
^{2}\right.  \rightarrow\left.  ^{\otimes}T_{\bullet}^{1}\right.
\rightarrow\left.  ^{\otimes}T_{\bullet}^{0}\right.  \rightarrow0$, so that
$(\left.  ^{\otimes}T_{\bullet}\right.  )^{\bullet}$ becomes an exact
superscript-indexed complex of (subscript-indexed complexes); and the same for
$\left.  ^{\wedge}T_{\bullet}^{\bullet}\right.  $. We begin by discussing
$\left.  ^{\otimes}T_{\bullet}^{\bullet}\right.  $.\newline We define a
$\mathfrak{g}$-module $N^{0}:=\mathfrak{g}$ and for $p\geq1$%
\begin{equation}
N^{p}:=\coprod\nolimits_{s_{1},\ldots,s_{n}\in\{+,-,0\}}I_{1}^{s_{1}}\cap
I_{2}^{s_{2}}\cap\cdots\cap I_{n}^{s_{n}}\text{\qquad(with }\deg(s_{1}%
,\ldots,s_{n})=p\text{).} \label{lBTA_12}%
\end{equation}
We shall denote the components $f=(f_{s_{1}\ldots s_{n}})$ of elements in
$N^{p}$ with indices in terms of $s_{1},\ldots,s_{n}\in\{+,-,0\}$. Clearly
$N^{p}=0$ for $p>n+1$. We shall treat all $N^{p}$ as $\mathfrak{g}$-modules
and observe that%
\[
\left.  ^{\otimes}T_{\bullet}^{p}\right.  =C(N^{p})_{\bullet}%
\]
(by definition!), so by the functoriality and flatness\footnote{We just tensor
$N^{p}$ with the vector spaces $%
{\textstyle\bigwedge\nolimits^{i}}
\mathfrak{g}$. Being over a field, this preserves exact sequences.} of
$C_{\bullet}$ it suffices to construct an exact complex $N^{\bullet}$ out of
the $N^{p}$ and then $\left.  ^{\otimes}T_{\bullet}^{p}\right.  $ will be an
exact complex in $p$.

\begin{example}
For $n=1$ we have%
\[
N^{2}=I_{1}^{0}\text{,\qquad}N^{1}=I_{1}^{+}\oplus I_{1}^{-}%
\]
and elements would be denoted $f=(f_{0})\in N^{2}$ and $g=(g_{+},g_{-})\in
N^{1}$. For $n=2$ we have%
\begin{align*}
N^{3}  &  =I_{1}^{0}\cap I_{2}^{0}\text{,\qquad}N^{2}=\left(
{\textstyle\coprod\nolimits_{s_{1}\in\{+,-\}}}
I_{1}^{s_{1}}\cap I_{2}^{0}\right)  \oplus\left(
{\textstyle\coprod\nolimits_{s_{2}\in\{+,-\}}}
I_{1}^{0}\cap I_{2}^{s_{2}}\right) \\
N^{1}  &  =%
{\textstyle\coprod\nolimits_{s_{1},s_{2}\in\{+,-\}}}
I_{1}^{s_{1}}\cap I_{2}^{s_{2}}\text{.}%
\end{align*}

\end{example}

We shall use the shorthand $s_{1}\ldots\pm\ldots s_{n}$ (resp. $0$ instead of
$\pm$) to indicate that $s_{i}\in\{+,-\}$ (resp. $s_{i}=0$) sits in the $i$-th
place. Define $\mathfrak{g}$-module homomorphisms%
\begin{align}
\left(  \partial_{i}f\right)  _{s_{1}\ldots\pm\ldots s_{n}}:=  &  \left(
-1\right)  ^{\#\left\{  j\mid j>i\text{ and }s_{j}=0\right\}  }f_{s_{1}%
\ldots0\ldots s_{n}}\nonumber\\
(\partial_{i}f)_{s_{1}\ldots0\ldots s_{n}}:=  &
0\label{lBT_revDifferentialEasyDef}\\
\partial:=  &
{\textstyle\sum\nolimits_{i=1}^{n}}
\partial_{i}\nonumber
\end{align}
One checks easily that $\partial_{i}^{2}=0$ and $\partial_{i}\partial
_{j}+\partial_{j}\partial_{i}=0$ for all $i,j=1,\ldots,n$. As a consequence,
$\partial^{2}=0$. The components are given explicitly by%
\begin{align}
\left(  \partial f\right)  _{s_{1}\ldots s_{n}}  &  =%
{\textstyle\sum\nolimits_{i=1}^{n}}
(\partial_{i}f)_{s_{1}\ldots s_{n}}\nonumber\\
&  =%
{\textstyle\sum\limits_{\{i\mid s_{i}=+,-\}}}
\left(  -1\right)  ^{\#\left\{  j\mid j>i\text{ and }s_{j}=0\right\}
}f_{s_{1}\ldots0\ldots s_{n}}\text{.} \label{lBT_revDiffCubeComplex1}%
\end{align}

\begin{definition}
\label{BT_Def_IdempotentsForA}Let $(A,(I_{i}^{\pm}),\tau)$ be an $n$-fold
cubically decomposed algebra over a field $k$. A \emph{system of good
idempotents} are pairwise commuting elements $P_{i}^{+}\in A$ for
$i=1,\ldots,n$ such that for all $i$:

\begin{enumerate}
\item $P_{i}^{+2}=P_{i}^{+}$.

\item $P_{i}^{+}A\subseteq I_{i}^{+}$.

\item $P_{i}^{-}A\subseteq I_{i}^{-}\qquad$(where we define $P_{i}%
^{-}:=\mathbf{1}_{A}-P_{i}^{+}$).
\end{enumerate}
\end{definition}

We note that the $P_{i}^{-}$ are also pairwise commuting idempotents and
$P_{i}^{+}+P_{i}^{-}=\mathbf{1}_{A}$. Next, for $s_{i}\in\{+,-\}$ define
$k$-vector space homomorphisms%
\begin{align*}
\left(  \varepsilon_{i}f\right)  _{s_{1}\ldots s_{i}\ldots s_{n}}:=  &
\left(  -1\right)  ^{s_{i}}P_{i}^{s_{i}}%
{\textstyle\sum\nolimits_{\gamma_{i}\in\{\pm\}}}
\left(  -1\right)  ^{\gamma_{i}}f_{s_{1}\ldots\gamma_{i}\ldots s_{n}}\\
\left(  \varepsilon_{i}f\right)  _{s_{1}\ldots0\ldots s_{n}}:=  &  0\text{,}%
\end{align*}
where $(-1)^{\pm}=\pm1$. By direct calculation one verifies the identities
$\varepsilon_{i}^{2}=\varepsilon_{i}$ and $\varepsilon_{i}\varepsilon
_{j}=\varepsilon_{j}\varepsilon_{i}$ for all $i,j=1,\ldots,n$. Finally, define%
\begin{align*}
\left(  H_{i}f\right)  _{s_{1}\ldots0\ldots s_{n}}  &  :=\left(  -1\right)
^{\#\left\{  j\mid j>i\text{ and }s_{j}=0\right\}  }%
{\textstyle\sum\nolimits_{\gamma_{i}\in\{\pm\}}}
P_{i}^{-\gamma_{i}}f_{s_{1}\ldots\gamma_{i}\ldots s_{n}}\\
\left(  H_{i}f\right)  _{s_{1}\ldots\pm\ldots s_{n}}  &  :=0\text{.}%
\end{align*}
The expression $P_{i}^{-\gamma_{i}}$ means $P_{i}^{-}$ for $\gamma_{i}=+$ and
$P_{i}^{+}$ for $\gamma_{i}=-$. One checks that%
\begin{align*}
H_{i}^{2}=0\qquad &  \text{and}\qquad H_{i}H_{j}+H_{j}H_{i}=0\\
\partial_{i}\varepsilon_{j}=\varepsilon_{j}\partial_{i}\qquad &
\text{and}\qquad H_{i}\varepsilon_{j}=\varepsilon_{j}H_{i}%
\end{align*}
for all $i,j=1,\ldots,n$. Moreover, $\partial_{i}H_{j}+H_{j}\partial_{i}=0$
whenever $i\neq j$. In the special case $i=j$ one finds instead that%
\[
\partial_{i}H_{i}+H_{i}\partial_{i}=\mathbf{1}-\varepsilon_{i}\text{.}%
\]
Define $H:=H_{1}+\varepsilon_{1}H_{2}+\cdots+\varepsilon_{1}\varepsilon
_{2}\cdots\varepsilon_{n-1}H_{n}$. Using the identities established above, one
finds very easily%
\begin{equation}
H^{2}=0\qquad\text{and}\qquad\partial H+H\partial=\mathbf{1}-\varepsilon
_{1}\cdots\varepsilon_{n}\text{.} \label{lBTA_35}%
\end{equation}
The fact $H^{2}=0$ was observed by the anonymous referee; it explains a
certain cancellation in the proof of Prop. \ref{BT_KeyLemmaFormula}, which had
been rather mysterious in an earlier version of this text.

\begin{lemma}
An explicit formula for $H$ is given by%
\begin{align}
(Hf)_{s_{1}\ldots s_{n}}  &  =\left(  -1\right)  ^{\deg(s_{1}\ldots s_{n}%
)}\left(  -1\right)  ^{s_{1}+\cdots+s_{b}}P_{1}^{s_{1}}\cdots P_{i}^{s_{b}%
}\label{lBTA_17}\\
&
{\textstyle\sum\limits_{\gamma_{1}\ldots\gamma_{b+1}\in\{\pm\}}}
\left(  -1\right)  ^{\gamma_{1}+\cdots+\gamma_{b}}P_{b+1}^{-\gamma_{b+1}%
}f_{\gamma_{1}\ldots\gamma_{b+1}s_{b+2}\ldots s_{n}}\text{,}\nonumber
\end{align}
where $b$ denotes the largest index such that $s_{1},\ldots,s_{b}\in\{\pm\}$
or $b=0$ if none (and so $s_{b+1}=0$ if $b<n$; $b+1$ is the index of the
\textquotedblleft leftmost zero\textquotedblright).
\end{lemma}

\begin{proof}
One shows that%
\begin{align}
\left(  \varepsilon_{1}\cdots\varepsilon_{i}f\right)  _{s_{1}\ldots s_{n}}  &
=\left(  -1\right)  ^{s_{1}+\cdots+s_{i}}P_{1}^{s_{1}}\cdots P_{i}^{s_{i}%
}\nonumber\\
&
{\textstyle\sum\limits_{\gamma_{1}\ldots\gamma_{i}\in\{\pm\}}}
\left(  -1\right)  ^{\gamma_{1}+\cdots+\gamma_{i}}f_{\gamma_{1}\ldots
\gamma_{i}s_{i+1}\ldots s_{n}}\label{lBTA_34}\\
&  \qquad\qquad\qquad\text{(for }s_{1},\ldots,s_{i}\in\{\pm\}\text{)}%
\nonumber\\
\left(  \varepsilon_{1}\cdots\varepsilon_{i}f\right)  _{s_{1}\ldots s_{n}}  &
=0\text{,}\qquad\text{(if }0\in\{s_{1},\ldots,s_{i}\}\text{)}\nonumber
\end{align}
by evaluating $(\varepsilon_{j}\cdots\varepsilon_{i}f)$ inductively along
$j=i,i-1,\ldots,1$. Plug in $H_{i+1}f$ for $f$ to obtain%
\begin{align*}
\left(  \varepsilon_{1}\cdots\varepsilon_{i}H_{i+1}f\right)  _{s_{1}\ldots
s_{n}}  &  =\left(  -1\right)  ^{\#\left\{  j\mid j>i+1\text{ and }%
s_{j}=0\right\}  }\left(  -1\right)  ^{s_{1}+\cdots+s_{i}}P_{1}^{s_{1}}\cdots
P_{i}^{s_{i}}\\
&
{\textstyle\sum\limits_{\gamma_{1}\ldots\gamma_{i+1}\in\{\pm\}}}
\left(  -1\right)  ^{\gamma_{1}+\cdots+\gamma_{i}}P_{i+1}^{-\gamma_{i+1}%
}f_{\gamma_{1}\ldots\gamma_{i}\gamma_{i+1}s_{i+2}\ldots s_{n}}%
\end{align*}
for $s_{1},\ldots,s_{i}\in\{\pm\}$ and $s_{i+1}=0$. Otherwise, i.e. for
$0\in\{s_{1},\ldots,s_{i}\}$ or $s_{i+1}\in\{\pm\}$, the respective component
is zero. Thus,%
\[
H_{s_{1}\ldots s_{n}}=%
{\textstyle\sum\nolimits_{i=1}^{n}}
\left(  \varepsilon_{1}\cdots\varepsilon_{i}H_{i+1}f\right)  _{s_{1}\ldots
s_{n}}\text{.}%
\]
The summands with $i>b$ vanish since for them $0\in\{s_{1},\ldots,s_{i}\}$.
The summands with $i<b$ vanish since for them $s_{i+1}\in\{\pm\}$. Thus,%
\[
H_{s_{1}\ldots s_{n}}=\left(  \varepsilon_{1}\cdots\varepsilon_{b}%
H_{b+1}f\right)  _{s_{1}\ldots s_{n}}%
\]
and we use the above explicit formula. Note that $\#\left\{  j\mid j>b+1\text{
and }s_{j}=0\right\}  $ is just one below the total number of slots with value
$0$ since $s_{1},\ldots,s_{b}\in\{\pm\}$ and $s_{b+1}=0$. Thus, $\left(
-1\right)  ^{\#\left\{  j\mid j>i+1\text{ and }s_{j}=0\right\}  }=\left(
-1\right)  ^{\deg(s_{1}\ldots s_{n})}$.
\end{proof}

The above maps are defined for $N^{p}$ in degrees $\geq1$. We extend them to
degree zero by%
\[
\hat{\partial}:N^{1}\rightarrow N^{0}\qquad\text{and}\qquad\hat{H}%
:N^{0}\rightarrow N^{1}%
\]%
\begin{align}
\hat{\partial}f  &  :=%
{\textstyle\sum\limits_{s_{1}\ldots s_{n}\in\{+,-\}}}
\left(  -1\right)  ^{s_{1}+\cdots+s_{n}}f_{s_{1}\ldots s_{n}}\nonumber\\
(\hat{H}f)_{s_{1}\ldots s_{n}}  &  :=(-1)^{s_{1}+\cdots+s_{n}}P_{1}^{s_{1}%
}\cdots P_{n}^{s_{n}}f\text{.} \label{lBTA_14}%
\end{align}
Along with these, we obtain the following crucial fact:

\begin{lemma}
\label{BT_Prop_EstablishKeyCubeComplex}Equipped with these morphisms%
\begin{equation}
N^{\bullet}=[N^{n+1}\underset{H}{\overset{\partial}{\rightleftarrows}}%
N^{n}\underset{H}{\overset{\partial}{\rightleftarrows}}\cdots\underset
{H}{\overset{\partial}{\rightleftarrows}}N^{1}\underset{\hat{H}}{\overset
{\hat{\partial}}{\rightleftarrows}}N^{0}]_{n+1,0} \label{lBTA_7}%
\end{equation}
is a complex of $\mathfrak{g}$-modules with differentials $\partial_{\bullet}$
(resp. $\hat{\partial}$) and contracting homotopies $H_{\bullet}$ (resp
$\hat{H}$) in the category of $k$-vector spaces.
\end{lemma}

\begin{proof}
The identities $\partial^{2}=0$ and $\hat{\partial}\circ\partial
=0:N^{2}\rightarrow N^{0}$ are easy to check. Next, we confirm the contracting
homotopy. We find $\partial H+H\partial=\mathbf{1}-\varepsilon_{1}%
\cdots\varepsilon_{n}$ by a telescope cancellation. For $f\in N^{i}$ with
$i\geq2$ for each component $f_{s_{1}\ldots s_{n}}$ there must be at least one
$i$ with $s_{i}=0$ and thus $\varepsilon_{1}\cdots\varepsilon_{n}\mid_{N^{i}%
}=0$ for $i\geq2$. It remains to treat $i=0,1$. For $i=1$ we compute%
\[
\hat{H}\hat{\partial}f=(-1)^{s_{1}+\cdots+s_{n}}P_{1}^{s_{1}}\cdots
P_{n}^{s_{n}}%
{\textstyle\sum\limits_{s_{1}\ldots s_{n}\in\{+,-\}}}
\left(  -1\right)  ^{s_{1}+\cdots+s_{n}}f_{s_{1}\ldots s_{n}}=\varepsilon
_{1}\cdots\varepsilon_{n}f
\]
(as in eq. \ref{lBTA_34}). Thus, $\partial H+\hat{H}\hat{\partial}=\mathbf{1}$
on $N^{1}$. Finally, for $i=0$ we compute $\hat{\partial}\hat{H}f=f$.
\end{proof}

\begin{corollary}
$0\rightarrow\left.  ^{\otimes}T_{\bullet}^{n+1}\right.  \rightarrow\left.
^{\otimes}T_{\bullet}^{n}\right.  \rightarrow\cdots\rightarrow\left.
^{\otimes}T_{\bullet}^{0}\right.  \rightarrow0$ with differential (and a
contracting homotopy) induced by $\partial\otimes\operatorname*{id}%
\nolimits_{\wedge^{\bullet}\mathfrak{g}}$ (and $H\otimes\operatorname*{id}%
\nolimits_{\wedge^{\bullet}\mathfrak{g}}$) is an exact complex of (complexes
of $k$-vector spaces).
\end{corollary}

For the corollary just use that tensoring with $%
{\textstyle\bigwedge\nolimits^{r}}
\mathfrak{g}$ is exact.

\section{\label{section_CubeComplexII}The cube complex II}

Next, it would be nice to give a discussion of the $\left.  ^{\wedge
}T_{\bullet}^{\bullet}\right.  $ parallel to the one for $\left.  ^{\otimes
}T_{\bullet}^{\bullet}\right.  $ in the previous section. We can only do this
to a limited extent, however.

\begin{lemma}
\label{BTrev_LemmaConstructTWedgeComplex}The definition%
\begin{equation}
\left(  \partial f\right)  _{s_{1}\ldots s_{n}}=%
{\textstyle\sum\limits_{\{i\mid s_{i}=+,-\}}}
\left(  -1\right)  ^{\#\left\{  j\mid j>i\text{ and }s_{j}=0\right\}
}f_{s_{1}\ldots0\ldots s_{n}} \label{lBT_revDiffCubeComplex2}%
\end{equation}
turns $\left.  ^{\wedge}T_{\bullet}^{\bullet}\right.  $ into a complex of
(complexes of $k$-vector spaces) with respect to the superscript index. The
morphisms $\left.  ^{\otimes}T_{\bullet}^{p}\right.  \overset{I}%
{\longrightarrow}\left.  ^{\wedge}T_{\bullet+1}^{p}\right.  $ yield a morphism
of complexes.
\end{lemma}

\begin{proof}
Easy. Just check that the map $\partial$ is well-defined and satisfies
$\partial^{2}=0$; in fact exactly the same computation as in eqs.
\ref{lBT_revDifferentialEasyDef} applies. For the second claim, we just need
to show that the map $I$ commutes with the differential of either complex, but
this is clear since the differentials are given by the same formula, compare
eq. \ref{lBT_revDiffCubeComplex1} with eq. \ref{lBT_revDiffCubeComplex2}.
\end{proof}

The complex $\left.  ^{\wedge}T_{\bullet}^{\bullet}\right.  $ is the central
object in Beilinson's construction \cite{MR565095}. We will use its analogue
$\left.  ^{\otimes}T_{\bullet}^{\bullet}\right.  $ as an auxiliary
computational device. Firstly, let us explain Beilinson's construction. We
need the following entirely homological tool:

\begin{lemma}
\label{BT_PropExplicitDifferentialInSpecSeq}Suppose we are given an exact
sequence%
\[
S^{\bullet}=[S^{n+1}\rightarrow S^{n}\rightarrow\cdots\rightarrow
S^{0}]_{n+1,0}%
\]
with entries in $\mathbf{Ch}^{+}\mathcal{M}od_{k}$, i.e. each $S^{i}%
=S_{\bullet}^{i}$ is a bounded below complex of $k$-vector spaces\footnote{One
may alternatively view this as a bicomplex supported horizontally in degrees
$[0,n+1]$, bounded from below, and whose rows are exact.}.

\begin{enumerate}
\item There is a second quadrant homological spectral sequence $(E_{p,q}%
^{r},d_{r})$ converging to zero such that%
\[
E_{p,q}^{1}=H_{q}(S_{\bullet}^{p})\text{.\qquad\emph{(}}d_{r}:E_{p,q}%
^{r}\rightarrow E_{p-r,q+r-1}^{r}\text{\emph{)}}%
\]

\item There is a first quadrant cohomological spectral sequence $(E_{r}%
^{p,q},d^{r})$ converging to zero such that%
\[
E_{1}^{p,q}=H^{q}(\operatorname*{Hom}\nolimits_{k}(S_{\bullet}^{p}%
,k))\text{.}\qquad\emph{(}d^{r}:E_{r}^{p,q}\rightarrow E_{r}^{p+r,q-r+1}%
\emph{)}%
\]

\item The following differentials are isomorphisms:%
\[
d_{n+1}:E_{n+1,1}^{n+1}\rightarrow E_{0,n+1}^{n+1}\qquad\text{and}\qquad
d^{n+1}:E_{n+1}^{0,n+1}\rightarrow E_{n+1}^{n+1,1}\,\text{.}%
\]

\item \label{lemma_specseq_homotopypart}Suppose $H_{p}:S^{p}\rightarrow
S^{p+1}$ is a contracting homotopy for $S^{\bullet}$. Then%
\[
(d_{n+1})^{-1}=H_{n}\delta_{1}H_{n-1}\cdots\delta_{n-1}H_{1}\delta_{n}%
H_{0}=H_{n}%
{\textstyle\prod\nolimits_{i=1,\ldots,n}}
(\delta_{i}H_{n-i})
\]
(where the last product depends on the ordering and refers to composition),
and%
\[
(d^{n+1})^{-1}=H_{0}^{\ast}\delta_{n}^{\ast}H_{1}^{\ast}\cdots\delta_{1}%
^{\ast}H_{n}^{\ast}=H_{0}^{\ast}%
{\textstyle\prod\nolimits_{i=n,\ldots,1}}
(\delta_{i}^{\ast}H_{n+1-i}^{\ast})\text{,}%
\]
where we write $f^{\ast}=\operatorname*{Hom}\nolimits_{k}(f,k)$ as a shorthand.
\end{enumerate}

The construction is functorial in $S^{\bullet}$, i.e. if $S^{\bullet
}\rightarrow S^{\prime\bullet}$ is a morphism of complexes as in our
assumptions, then there are induced morphisms between their spectral sequences.
\end{lemma}

\begin{proof}
Parts (1)-(3) are \cite[Lemma 1(a)]{MR565095}. More precisely, for
\textbf{(1)} use the bicomplex spectral sequence for%
\[
E_{p,q}^{0}=S_{q}^{p}\text{\quad and\quad}E_{0}^{p,q}=\operatorname*{Hom}%
\nolimits_{k}(S_{q}^{p},k)\text{.}%
\]
If we take differentials `$\rightarrow$' for forming the $E^{0}$-page, the
$E^{1}$-page vanishes since $S_{\bullet}$ is exact (as a complex of complexes)
and so the individual sequences of $k$-vector spaces $S_{\bullet}^{i}$ for
constant $i$ are exact, so $E^{\infty}=E^{1}=0$. Then use the bicomplex
spectral sequences with differential `$\downarrow$' on the $E^{0}$-page for
our claim. It also converges to zero then; \textbf{(2)} is analogous.
\textbf{(3)} The bicomplex is horizontally supported in $[0,n+1]$.
\textbf{(4)}\ Diagram chase.
\end{proof}

We combine Lemma \ref{BTrev_LemmaConstructTWedgeComplex}\ with Lemma
\ref{BT_PropExplicitDifferentialInSpecSeq}: Apply the latter to $S_{q}%
^{p}:=\left.  ^{\wedge}T_{q}^{p}\right.  $; we denote the resulting spectral
sequence by $\left.  ^{\wedge}E_{\bullet,\bullet}^{\bullet}\right.  $. The
fact that the (bi)complex of Lemma \ref{BT_PropExplicitDifferentialInSpecSeq}
is supported horizontally in $[n+1,0]$ (homologically, i.e. for $\left.
^{\wedge}E_{\bullet,\bullet}^{\bullet}\right.  $) and $[0,n+1]$ respectively
(cohomologically, i.e. for $\left.  ^{\wedge}E_{\bullet}^{\bullet,\bullet
}\right.  $) implies that we have edge morphisms%
\begin{align*}
\rho_{1}:\left.  ^{\wedge}E_{n+1,1}^{n+1}\right.  \rightarrow\left.  ^{\wedge
}E_{n+1,1}^{1}\right.  \text{\qquad}  &  \text{and\qquad}\rho_{2}:\left.
^{\wedge}E_{0,n+1}^{1}\right.  \rightarrow\left.  ^{\wedge}E_{0,n+1}%
^{n+1}\right. \\
\wp_{1}:\left.  ^{\wedge}E_{n+1}^{0,n+1}\right.  \rightarrow\left.  ^{\wedge
}E_{1}^{0,n+1}\right.  \text{\qquad}  &  \text{and\qquad}\wp_{2}:\left.
^{\wedge}E_{1}^{n+1,1}\right.  \rightarrow\left.  ^{\wedge}E_{n+1}%
^{n+1,1}\right.  \text{.}%
\end{align*}
Next, we identify the involved objects: Using Lemma
\ref{BT_LemmaOnComputingLieHomology} we compute%
\begin{align*}
\left.  ^{\wedge}E_{0,n+1}^{1}\right.   &  =H_{n+1}(\left.  ^{\wedge
}T_{\bullet}^{0}\right.  )=H_{n+1}(CE(\mathfrak{g})_{\bullet})\cong
H_{n+1}(\mathfrak{g},k)\\
\left.  ^{\wedge}E_{n+1,1}^{1}\right.   &  =H_{1}(\left.  ^{\wedge}T_{\bullet
}^{n+1}\right.  )=H_{1}(%
{\textstyle\bigcap\nolimits_{i=1,\ldots,n}}
{\textstyle\bigcap\nolimits_{s_{i}\in\{\pm\}}}
CE(I_{i}^{s_{i}})_{\bullet})=I_{\operatorname*{tr}}/[I_{\operatorname*{tr}%
},\mathfrak{g}]\\
\left.  ^{\wedge}E_{1}^{n+1,1}\right.   &  =\operatorname*{Hom}\nolimits_{k}%
(I_{\operatorname*{tr}}/[I_{\operatorname*{tr}},\mathfrak{g}],k)\qquad
\text{and}\qquad\left.  ^{\wedge}E_{1}^{0,n+1}\right.  =H^{n+1}(\mathfrak{g}%
,k)\text{.}%
\end{align*}

\begin{definition}
[\cite{MR565095}]\label{BT_Def_TateBeilinsonAbstractResidueMaps}Let
$(A,(I_{i}^{\pm}),\tau)$ be an $n$-fold cubically decomposed algebra over a
field $k$ and $\mathfrak{g}:=A_{Lie}$ its Lie algebra. Define%
\[
\operatorname*{res}\nolimits_{\ast}:H_{n+1}(\mathfrak{g},k)\rightarrow
k\qquad\operatorname*{res}\nolimits_{\ast}:=\tau\circ\rho_{1}\circ
(d_{n+1})^{-1}\circ\rho_{2}%
\]
and%
\[
\operatorname*{res}\nolimits^{\ast}:k\rightarrow H^{n+1}(\mathfrak{g}%
,k)\quad\operatorname*{res}\nolimits^{\ast}(1):=(\wp_{1}\circ(d^{n+1}%
)^{-1}\circ\wp_{2})\tau\text{,}%
\]
where for $\operatorname*{res}\nolimits^{\ast}$ we read $\tau$ as an element
of $E_{1}^{n+1,1}$. We will call $\phi:=\operatorname*{res}\nolimits^{\ast
}(1)$ the \emph{Tate extension class}.
\end{definition}

In the case $n=1$ it would also be justified to name this cohomology class
after\ Kac-Petersen \cite{MR619827}; it also appears in the works of the
Japanese school, e.g. \cite{MR723457}.

\begin{remark}
It follows from the construction of $\operatorname*{res}\nolimits_{\ast}$,
$\operatorname*{res}\nolimits^{\ast}$ that%
\begin{equation}
\operatorname*{res}\nolimits^{\ast}(\alpha)(X_{0}\wedge\ldots\wedge
X_{n})=\alpha\operatorname*{res}\nolimits_{\ast}X_{0}\wedge\ldots\wedge
X_{n}\text{.} \label{lBT_39}%
\end{equation}

\end{remark}

Now we would like to compute these maps explicitly. Clearly, the most elusive
map in the construction is the differential $d_{n+1}$ (resp. $d^{n+1}$). We
can render it explicit using Lemma \ref{BT_PropExplicitDifferentialInSpecSeq}%
.\ref{lemma_specseq_homotopypart} as soon as we have an explicit contracting
homotopy available. However, it seems to be quite difficult to construct such
a homotopy for the complex $\left.  ^{\wedge}T^{\bullet}\right.  $. On the
other hand, we \textit{do} have such a contracting homotopy for $\left.
^{\otimes}T^{\bullet}\right.  $ by Lemma \ref{BT_Prop_EstablishKeyCubeComplex}
and its corollary. Luckily for us, these complexes are closely connected. We
may apply Lemma \ref{BT_PropExplicitDifferentialInSpecSeq} also to $S_{q}%
^{p}:=\left.  ^{\otimes}T_{q-1}^{p}\right.  $; this time denote the resulting
spectral sequence by $\left.  ^{\otimes}E_{\bullet,\bullet}^{\bullet}\right.
$. We easily compute
\begin{align*}
\left.  ^{\otimes}E_{0,n+1}^{1}\right.   &  =H_{n+1}(\left.  ^{\otimes
}T_{\bullet-1}^{0}\right.  )=H_{n}(C(\mathfrak{g})_{\bullet})\cong
H_{n}(\mathfrak{g},\mathfrak{g})\\
\left.  ^{\otimes}E_{n+1,1}^{1}\right.   &  =H_{1}(\left.  ^{\otimes
}T_{\bullet-1}^{n+1}\right.  )=H_{0}(C(%
{\textstyle\bigcap\nolimits_{i=1,\ldots,n}}
{\textstyle\bigcap\nolimits_{s_{i}\in\{\pm\}}}
I_{i}^{s_{i}})_{\bullet})=I_{\operatorname*{tr}}/[I_{\operatorname*{tr}%
},\mathfrak{g}]\\
\left.  ^{\otimes}E_{1}^{n+1,1}\right.   &  =\operatorname*{Hom}%
\nolimits_{k}(I_{\operatorname*{tr}}/[I_{\operatorname*{tr}},\mathfrak{g}%
],k)\qquad\text{and}\qquad\left.  ^{\otimes}E_{1}^{0,n+1}\right.
=H^{n}(\mathfrak{g},\mathfrak{g}^{\ast})\text{.}%
\end{align*}
We note that some groups even agree with their $\left.  ^{\wedge}T_{q}%
^{p}\right.  $-counterpart; as we had already observed in eq.
\ref{lBTrev_curious_identities}.

\begin{definition}
Write $\left.  ^{\otimes}\operatorname*{res}\nolimits_{\ast}\right.
:H_{n}(\mathfrak{g},\mathfrak{g})\rightarrow k$ and $\left.  ^{\otimes
}\operatorname*{res}\nolimits^{\ast}(1)\right.  \in H^{n}(\mathfrak{g}%
,\mathfrak{g}^{\ast})$ for the counterparts of $\left.  \operatorname*{res}%
\nolimits_{\ast}\right.  ,\left.  \operatorname*{res}\nolimits^{\ast}\right.
$ in Def. \ref{BT_Def_TateBeilinsonAbstractResidueMaps} using $\left.
^{\otimes}E\right.  $ instead of $\left.  ^{\wedge}E\right.  $.
\end{definition}

\begin{lemma}
[Compatibility]\label{BTrev_ComputeWedgeResViaTensorRes}The morphism of
bicomplexes $\left.  ^{\otimes}T_{\bullet}^{\bullet}\right.  \overset
{I}{\longrightarrow}\left.  ^{\wedge}T_{\bullet+1}^{\bullet}\right.  $ induces
a commutative diagram%
\[
\xymatrix{ H_{n}(\mathfrak{g},\mathfrak{g}) \ar[d] \ar[r] \ar@/^2pc/[rrr]^{\emph{comes with contracting homotopy}} & {\left. ^{\otimes }E^{n+1}_{0,n+1}\right.} \ar[d] & {\left. ^{\otimes }E^{n+1}_{n+1,1}\right.} \ar[d] \ar[l]^{d_{n+1}}_{\cong} \ar[r] & H_{0}(\mathfrak{g},\mathfrak{g}) \ar[d]^{\cong} \\ H_{n+1}(\mathfrak{g},k) \ar[r] \ar@/_2pc/[rrr]_{\emph{Beilinson's residue}} & {\left. ^{\wedge }E^{n+1}_{0,n+1}\right.} & {\left. ^{\wedge }E^{n+1}_{n+1,1}\right.} \ar[l]_{d_{n+1}}^{\cong} \ar[r] & H_{1}(\mathfrak{g},k). }
\]

\end{lemma}

\begin{proof}
We had already observed in Lemma \ref{BTrev_LemmaConstructTWedgeComplex} that
the morphisms $I$ induce a morphism of bicomplexes. The spectral sequences
$\left.  ^{\otimes}E_{\bullet,\bullet}^{\bullet}\right.  $ and $\left.
^{\wedge}E_{\bullet,\bullet}^{\bullet}\right.  $ both arise from Lemma
\ref{BT_PropExplicitDifferentialInSpecSeq}, so by the functoriality of the
construction we get an induced morphism of spectral sequences. In particular,
all squares%
\[
\xymatrix{ {\left. ^{\otimes }E^{r}_{p,q}\right.} \ar[r]^-{d_{r}} \ar[d] & {\left. ^{\otimes }E^{r}_{p-r,q+r-1}\right.} \ar[d] \\ {\left. ^{\wedge }E^{r}_{p,q}\right.} \ar[r]_-{d_{r}} & {\left. ^{\wedge }E^{r}_{p-r,q+r-1}\right.} }
\]
commute, giving the middle square in our claim. The same applies to the edge
maps, giving the outer squares.
\end{proof}

Absolutely analogously we obtain a cohomological counterpart,%
\[
\xymatrix{
H^{1}(\mathfrak{g},k) \ar[r] \ar[d]_{\cong} &
H^{n+1}(\mathfrak{g},k) \ar[d] \\
H^{0}(\mathfrak{g},\mathfrak{g}^{\ast }) \ar[r] &
H^{n}(\mathfrak{g},\mathfrak{g}^{\ast }),
}
\]
where we have a contracting homotopy for the lower row. We leave the details
of this formulation to the reader.

\section{\label{section_ConcreteFormalism}Concrete Formalism}

Let $(A,(I_{i}^{\pm}),\tau)$ be an $n$-fold cubically decomposed algebra over
a field $k$. In \S \ref{section_CubeComplexII} we have constructed a canonical
morphism%
\[%
\begin{array}
[c]{cccc}%
\operatorname*{res}\nolimits_{\ast}: & H_{n+1}(\mathfrak{g},k) &
\longrightarrow & k\\
& \uparrow &  & \\
& H_{n}(\mathfrak{g},\mathfrak{g})\text{,} &  &
\end{array}
\]
where $\mathfrak{g}:=A_{Lie}$ is the Lie algebra associated to $A$. By Lemma
\ref{BTrev_ComputeWedgeResViaTensorRes}, its values on the image of
$H_{n}(\mathfrak{g},\mathfrak{g})\rightarrow H_{n+1}(\mathfrak{g},k)$ can be
computed via $\left.  ^{\otimes}\operatorname*{res}\nolimits_{\ast}\right.  $.
In this section we will obtain an explicit formula for the latter morphism.

Given the definition of $\left.  ^{\otimes}\operatorname*{res}\nolimits_{\ast
}\right.  $, Lemma \ref{BT_PropExplicitDifferentialInSpecSeq}%
.\ref{lemma_specseq_homotopypart} tells us that it can be given explicitly in
terms of differentials of the ordinary Chevalley-Eilenberg complexes
$C(-)_{\bullet}$ (cf. \S \ref{section_CEComplexes}) and contracting homotopies
of the cube complex $N^{\bullet}$ (cf. Lemma
\ref{BT_Prop_EstablishKeyCubeComplex} and its corollary), namely%
\begin{equation}
\left.  ^{\otimes}\operatorname*{res}\nolimits_{\ast}\right.  =\tau\circ
\rho_{1}\circ(^{\otimes}d_{n+1})^{-1}\circ\rho_{2}=\tau\circ\rho_{1}H%
{\textstyle\prod\nolimits_{i=1,\ldots,n}}
(\delta_{i}H)\rho_{2} \label{lBTA_20}%
\end{equation}
via the spectral sequence $\left.  ^{\otimes}E_{\bullet,\bullet}^{\bullet
}\right.  $. The contracting homotopy $H$ depends on the choice of a good
system of idempotents, see Def. \ref{BT_Def_IdempotentsForA}. Different
choices will yield formulas that may look different, but as $\left.
^{\otimes}\operatorname*{res}\nolimits_{\ast}\right.  $ (just like $\left.
\operatorname*{res}\nolimits_{\ast}\right.  $ itself) was defined entirely
independently of the choice of any idempotents, all such formulas actually
must agree.

Suppose a representative $\theta:=f_{0}\otimes f_{1}\ldots\wedge f_{n}$ with
$f_{0},\ldots,f_{n}\in N^{0}$ is given (note that $N^{0}$ equals
$\mathfrak{g}$ as a left-$U\mathfrak{g}$-module by definition, so it is valid
to treat all $f_{i}$ on equal footing). We shall compute $\left.  ^{\otimes
}\operatorname*{res}\nolimits_{\ast}\right.  \theta$ in several steps,
starting with $\theta_{0,n}:=\rho_{2}\theta$, then following%
\begin{equation}%
\begin{tabular}
[c]{ccccccc|c}
&  &  &  &  &  & $0$ & \\
&  &  &  &  &  & $\mid$ & \\
&  &  &  & $\theta_{1,n}$ & $\overset{H}{\longleftarrow}$ & $\theta_{0,n}$ &
$n$\\
&  &  &  & $\vdots$ &  &  & $\vdots$\\
&  & $\theta_{n,1}$ & $\overset{H}{\longleftarrow}$ & $\theta_{n-1,1}$ &  &  &
$1$\\
&  & $\downarrow$ &  &  &  &  & \\
$\theta_{n+1,0}$ & $\overset{H}{\longleftarrow}$ & $\theta_{n,0}$ &  &  &  &
& $0$\\\hline
$n+1$ &  & $n$ &  & $n-1$ & $\cdots$ & $0$ &
\end{tabular}
\ \ \qquad%
\begin{array}
[c]{ccc}
&  & q\\
&  & \uparrow\\
p & \leftarrow & +
\end{array}
\label{BT_FigLiftingTheThetaElements}%
\end{equation}
as prescribed by eq. \ref{lBTA_20}. This graphical arrangement elucidates the
position of the term of each step in the computation in the spectral sequence
from which eq. \ref{lBTA_20} originates, see Lemma
\ref{BT_PropExplicitDifferentialInSpecSeq}. However, for us each $\theta
_{\ast,\ast}$ will be an $E^{0}$-page representative of the respective
$E^{\ast}$-page term. Finally $\left.  ^{\otimes}\operatorname*{res}%
\nolimits_{\ast}\right.  \theta=\tau\rho_{1}\theta_{n+1,0}$. We note that
$\rho_{1},\rho_{2}$ are just edge maps, i.e. an inclusion of a subobject and a
quotient surjection. Hence, as we work with explicit representatives anyway,
the operation of these maps is essentially invisible (e.g. in the quotient
case it just means that our representative generates a larger equivalence class).

We will need a convenient notation for elements of this complex.\newline%
\textit{(Notation A)} We will write $\theta_{p,q-p\mid s_{1}\ldots s_{n}%
}^{w_{1}\ldots w_{p}}\in N^{p}$ for the summands in any expression of the
shape%
\begin{equation}
\theta_{p,q-p}=\sum_{\substack{w_{1}\ldots w_{p}\\\in\{1,\ldots,n\}}%
}\sum_{s_{1}\ldots s_{n}}\theta_{p,q-p\mid s_{1}\ldots s_{n}}^{w_{1}\ldots
w_{p}}\otimes f_{1}\wedge\ldots\wedge\widehat{f_{w_{1}}}\wedge\ldots
\wedge\widehat{f_{w_{p}}}\wedge\ldots\wedge f_{n}\text{,} \label{lBTA_15}%
\end{equation}
where

\begin{itemize}
\item $(p,q-p)$ denotes the location of the element in the bicomplex as in
fig. \ref{BT_FigLiftingTheThetaElements},

\item $s_{1},\ldots,s_{n}\in\{0,+,-\}$ denotes the component (= direct
summand) of $N^{p}$ as in eq. \ref{lBTA_12}, $f_{1},\ldots,f_{n}%
\in\mathfrak{g}$,

\item the additional superscripts $w_{1},\ldots,w_{p}\in\{1,\ldots,n\}$ are
used to indicate the omission of wedge factors.
\end{itemize}

Note that the values $\theta_{p,q\mid s_{1}\ldots s_{n}}^{w_{1}\ldots w_{p}}$
are not necessarily uniquely determined since the individual wedge tails need
not be linearly independent.\newline\textit{(Notation B)} We also need a
shorthand for the summands in any expression of the shape%

\begin{align}
\theta_{p,q-p-1}  &  =\sum_{\substack{w_{1}\ldots w_{p},w_{a},w_{b}%
\\\in\{1,\ldots,n\}}}\sum_{s_{1}\ldots s_{n}}\theta_{p,q\mid s_{1}\ldots
s_{n}}^{w_{1}\ldots w_{p}\parallel w_{a},w_{b}}\label{lBTA_21}\\
&  \otimes\lbrack f_{w_{a}},f_{w_{b}}]\wedge f_{1}\wedge\ldots\widehat
{f_{w_{1}}}\ldots\widehat{f_{w_{a}}}\ldots\widehat{f_{w_{b}}}\ldots
\widehat{f_{w_{p}}}\ldots\wedge f_{n}\text{.}\nonumber
\end{align}
Again $s_{1},\ldots,s_{n}$ denotes the component in $N^{p}$, $w_{1}%
,\ldots,w_{p}$ omitted wedge factors. Moreover, $w_{a}$ and $w_{b}$ denote two
additional omitted wedge factors and simultaneously indicate that $[f_{w_{a}%
},f_{w_{b}}]$ appears as an additional wedge factor. As for the previous
notation, the elements $\theta_{p,q\mid s_{1}\ldots s_{n}}^{w_{1}\ldots
w_{p}\parallel w_{a},w_{b}}\in N^{p}$ are not uniquely determined. We will
explain how these expressions arise soon.

\textit{Combinatorial Preparation:} We define for arbitrary $1\leq p\leq n$
and $w_{1},\ldots,w_{p}\in\{1,\ldots,n\}$ the `sign function' (a
generalization of the signum of a permutation)%
\begin{equation}
\rho(w_{1},\ldots,w_{p}):=\left(  -1\right)  ^{\sum_{k=1}^{p}\sum_{j<k}%
\delta_{w_{j}<w_{k}}}\text{.} \label{lBTA_19}%
\end{equation}
By abuse of language we do not carry the value $p$ in the notation for $\rho$
as it will always be clear from the number of arguments which variant is used.
It is easy to see that $\rho(w_{1})=+1$ and $\rho(w_{1},w_{2})=(-1)^{\delta
_{w_{1}<w_{2}}}$. For $p=n$ we have%
\begin{equation}
\rho(w_{1},\ldots,w_{n})=\operatorname*{sgn}\left(
\begin{array}
[c]{ccc}%
1 & \cdots & n\\
w_{1} & \cdots & w_{n}%
\end{array}
\right)  \text{.} \label{lBTA_22}%
\end{equation}
We shall need the inductive formula (which is easy to check by induction)%
\begin{equation}
(-1)^{\#\{w_{i}\mid1\leq i\leq p\text{ s.t. }w_{i}<w_{p+1}\}}\rho(w_{1}%
,\ldots,w_{p})=\rho(w_{1},\ldots,w_{p+1})\text{.} \label{lBTA_18}%
\end{equation}

\begin{proposition}
\label{BT_KeyLemmaFormula}Suppose $\theta:=f_{0}\otimes f_{1}\wedge
\ldots\wedge f_{n}$ with $f_{i}\in N_{0}=\mathfrak{g}$. Moreover, suppose
$P_{1}^{+},\ldots,P_{n}^{+}$ is a good system of idempotents as in Def.
\ref{BT_Def_IdempotentsForA}. Then for every $p\geq0$ the element
$\theta_{p+1,q}$ is of the shape as in eq. \ref{lBTA_15} and for $\gamma
_{1}\ldots\gamma_{n-p}\in\{+,-\}$ we have%
\begin{align*}
\theta_{p+1,q\mid\gamma_{1}\ldots\gamma_{n-p}\underset{p}{\underbrace
{0\ldots0}}}^{w_{1}\ldots w_{p}}  &  =(-1)^{\sum_{u=1}^{p-1}(u+1)}\left(
-1\right)  ^{w_{1}+\cdots+w_{p}}\rho(w_{1},\ldots,w_{p})\\
&  \left(  -1\right)  ^{\gamma_{1}+\cdots+\gamma_{n-p}}P_{1}^{\gamma_{1}%
}\cdots P_{n-p}^{\gamma_{n-p}}\\
&
{\textstyle\sum_{\gamma_{n-p+1}^{\ast}\ldots\gamma_{n}^{\ast}\in\{\pm\}}}
\left(  -1\right)  ^{\gamma_{n-p+1}^{\ast}+\cdots+\gamma_{n}^{\ast}}\\
&  \left(  P_{n-p+1}^{\left(  -\gamma_{n-p+1}^{\ast}\right)  }%
\operatorname*{ad}(f_{w_{p}})P_{n-p+1}^{\gamma_{n-p+1}^{\ast}}\right) \\
&  \cdots\left(  P_{n}^{\left(  -\gamma_{n}^{\ast}\right)  }\operatorname*{ad}%
(f_{w_{1}})P_{n}^{\gamma_{n}^{\ast}}\right)  f_{0}\text{.}%
\end{align*}
Here $\rho(w_{1},\ldots,w_{p})$ is the sign function defined in eq.
\ref{lBTA_19}. For $p=0$ the expression $\rho(w_{1},\ldots,w_{p})$ and the
whole sum $(\Sigma_{\{\pm\}}(\cdots))$ in $(\Sigma_{\{\pm\}}(\cdots))f_{0}$
should be read as $+1$ (giving the right-hand side of eq. \ref{lBTA_29} below).
\end{proposition}

\begin{itemize}
\item Note that no terms of the shape as in eq. \ref{lBTA_21} appear. This is
not entirely obvious in view of the definition of $\delta^{\lbrack2]}$, see
eq. \ref{lBT_CEComplexDifferential}.

\item The formula does not compute $\theta_{p+1,q\mid s_{1}\ldots s_{n}%
}^{w_{1}\ldots w_{p}}$ for arbitrary $s_{1}\ldots s_{n}$ of degree $p+1$. This
is due to the fact that we only have further use for the ones treated.

\item For $p\leq1$ read $\sum_{u=1}^{p-1}(u+1)$ as zero.
\end{itemize}

\begin{proof}
We prove this by induction. For $p=0$ the claim reads%
\begin{equation}
\theta_{1,q\mid\gamma_{1}\ldots\gamma_{n}}=\left(  -1\right)  ^{\gamma
_{1}+\cdots+\gamma_{n}}P_{1}^{\gamma_{1}}\cdots P_{n}^{\gamma_{n}}f_{0}
\label{lBTA_29}%
\end{equation}
and in view of eq. \ref{lBTA_14} this proves the claim in this case. Now we
proceed by induction. Assume the case $p$ is settled, i.e. in the notation of
eq. \ref{lBTA_15} $\theta_{p+1,q\mid\gamma_{1}\ldots\gamma_{n-p}\underset
{p}{\underbrace{0\ldots0}}}^{w_{1}\ldots w_{p}}$ is exactly as in our claim.
Next, we need to apply the differential $\delta_{q}=\delta_{q}^{[1]}%
+\delta_{q}^{[2]}$ of the Chevalley-Eilenberg resolution, see eq.
\ref{lBT_CEComplexDifferential}. The contribution of $\delta_{q}^{[1]}$ will
be relevant, but for $\delta_{q}^{[2]}$ we shall see that (after applying the
next contracting homotopy) the contribution vanishes.\ We treat each
$\delta^{\lbrack i]}$, $i=1,2$ separately:\newline\textbf{(1)} Consider
$\delta_{q}^{[1]}$ in eq. \ref{lBT_CEComplexDifferential}. The sum $\Sigma
_{i}$ \textit{loc. cit.} maps components indexed by $w_{1},\ldots,w_{p}$ to
components of $\delta^{\lbrack1]}\theta_{p,q}$, indexed by $w_{1},\ldots
,w_{p}$ and an additional $w_{p+1}\in\{1,\ldots,n\}\setminus\{w_{1}%
,\ldots,w_{p}\}$ -- they correspond to the summands of $\delta^{\lbrack
1]}\theta_{p,q}$ and to the additional omitted wedge factor respectively.
Moreover, the formula imposes signs $(-1)^{i+1}$, but here $i$ depends on the
numbering of the wedges $(\ldots\wedge\ldots\wedge\ldots)$. In the notation of
eq. \ref{lBTA_15} the subscript $j$ of $f_{j}$ does not necessarily indicate
the $f_{j}$ sits in the $j$-th wedge, due to the possible omission of wedge
factors $f_{w_{1}},\ldots,f_{w_{p}}$ on the left-hand side of it. To
compensate for that in the following computation the term $(-1)^{\#\{w_{i}%
\mid1\leq i\leq p\text{ s.t. }w_{i}<w_{p+1}\}}$ appears, sign-counting the
omission on the left of the new-to-be-omitted $w_{p+1}$ in the component of
$\delta^{\lbrack1]}\theta_{p+1,q}$. As $p$ remains constant, the indexing
$\gamma_{1}\ldots\gamma_{n-p}0\ldots0$ remains unaffected. We get for
$(\delta^{\lbrack1]}\theta_{p+1,q})_{p+1,q-1\mid\gamma_{1}\ldots\gamma
_{n-p}\underset{p}{\underbrace{0\ldots0}}}^{w_{1}\ldots w_{p}w_{p+1}}$ the
expression%
\begin{align*}
&  =(-1)^{\sum_{u=1}^{p-1}(u+1)}(-1)^{w_{p+1}+1}(-1)^{\#\{w_{i}\mid1\leq i\leq
p\text{ s.t. }w_{i}<w_{p+1}\}}\operatorname*{ad}(f_{w_{p+1}})\\
&  \left(  -1\right)  ^{w_{1}+\cdots+w_{p}}\rho(w_{1},\ldots,w_{p})\\
&  \left(  -1\right)  ^{\gamma_{1}+\cdots+\gamma_{n-p}}P_{1}^{\gamma_{1}%
}\cdots P_{n-p}^{\gamma_{n-p}}\\
&
{\textstyle\sum_{\gamma_{n-p+1}^{\ast}\ldots\gamma_{n}^{\ast}\in\{\pm\}}}
\left(  -1\right)  ^{\gamma_{n-p+1}^{\ast}+\cdots+\gamma_{n}^{\ast}}\\
&  \left(  P_{n-p+1}^{\left(  -\gamma_{n-p+1}^{\ast}\right)  }%
\operatorname*{ad}(f_{w_{p}})P_{n-p+1}^{\gamma_{n-p+1}^{\ast}}\right)
\cdots\left(  P_{n}^{\left(  -\gamma_{n}^{\ast}\right)  }\operatorname*{ad}%
(f_{w_{1}})P_{n}^{\gamma_{n}^{\ast}}\right)  f_{0}\text{.}%
\end{align*}
Next, we need to apply the contracting homotopy $H:N^{p+1}\rightarrow N^{p+2}%
$. Note that we have $p+1\geq1$, so eq. \ref{lBTA_17} applies. Note that for
indices $\gamma_{1}^{\dag}\ldots\gamma_{n-p-1}^{\dag}\underset{p+1}%
{\underbrace{0\ldots0}}$ with $\gamma_{1}^{\dag}\ldots\gamma_{n-p-1}^{\dag}%
\in\{\pm\}$ (i.e. indices of degree $p+2$, cf. eq. \ref{lBTA_12}) the index
$\gamma_{1}^{\dag}\ldots\gamma_{n-p-1}^{\dag}\underset{p}{\underbrace
{0\ldots0}}$ has degree $p+1$. The latter have been computed above. We obtain
for%
\[
(H\delta^{\lbrack1]}\theta_{p+1,q})_{p+2,q-1\mid\gamma_{1}^{\dag}\ldots
\gamma_{n-p-1}^{\dag}\underset{p+1}{\underbrace{0\ldots0}}}^{w_{1}\ldots
w_{p}w_{p+1}}%
\]
the expression
\begin{align*}
&  =(-1)^{p}(-1)^{\gamma_{1}^{\dag}+\cdots+\gamma_{n-p-1}^{\dag}}P_{1}%
^{\gamma_{1}^{\dag}}\cdots P_{n-p-1}^{\gamma_{n-p-1}^{\dag}}\\
&
{\textstyle\sum_{\gamma_{1},\ldots,\gamma_{(n-p-1)+1}\in\{\pm\}}}
(-1)^{\gamma_{1}+\cdots+\gamma_{n-p-1}}P_{(n-p-1)+1}^{-\gamma_{(n-p-1)+1}}\\
&  (\delta\theta_{p+1,q})_{p+1,q-1\mid\gamma_{1}\cdots\gamma_{n-p}\underset
{p}{\underbrace{0\ldots0}}}^{w_{1}\ldots w_{p+1}}\text{.}%
\end{align*}
In principle the first factor is $\left(  -1\right)  ^{\deg(\ldots
)}=(-1)^{p+2}$, but switching to $p$ preserves the correct sign. Next, we
expand this using our previous computation and obtain (by noting that many
signs are squares and thus $+1$)%
\begin{align*}
&  =(-1)^{\sum_{u=1}^{p-1}(u+1)}\left(  -1\right)  ^{p+1}\\
&  (-1)^{\gamma_{1}^{\dag}+\cdots+\gamma_{n-p-1}^{\dag}}(-1)^{\#\{w_{i}%
\mid1\leq i\leq p\text{ s.t. }w_{i}<w_{p+1}\}}\\
&  \left(  -1\right)  ^{w_{1}+\cdots+w_{p+1}}\rho(w_{1},\ldots,w_{p}%
)P_{1}^{\gamma_{1}^{\dag}}\cdots P_{n-p-1}^{\gamma_{n-p-1}^{\dag}}%
{\textstyle\sum_{\gamma_{n-p}\in\{\pm\}}}
(-1)^{\gamma_{n-p}}\\
&  \left(
{\textstyle\sum_{\gamma_{1},\ldots,\gamma_{n-p-1}\in\{\pm\}}}
P_{1}^{\gamma_{1}}\cdots P_{n-p-1}^{\gamma_{n-p-1}}\right)  P_{n-p}%
^{-\gamma_{n-p}}\operatorname*{ad}(f_{w_{p+1}})P_{n-p}^{\gamma_{n-p}}\\
&
{\textstyle\sum_{\gamma_{n-p+1}^{\ast}\ldots\gamma_{n}^{\ast}\in\{\pm\}}}
\left(  -1\right)  ^{\gamma_{n-p+1}^{\ast}+\cdots+\gamma_{n}^{\ast}}\\
&  \left(  P_{n-p+1}^{\left(  -\gamma_{n-p+1}^{\ast}\right)  }%
\operatorname*{ad}(f_{w_{p}})P_{n-p+1}^{\gamma_{n-p+1}^{\ast}}\right)
\cdots\left(  P_{n}^{\left(  -\gamma_{n}^{\ast}\right)  }\operatorname*{ad}%
(f_{w_{1}})P_{n}^{\gamma_{n}^{\ast}}\right)  f_{0}\text{.}%
\end{align*}
The sum in parantheses is the identity since for all $i$ we have $P_{i}%
^{+}+P_{i}^{-}=\mathbf{1}$ by Def. \ref{BT_Def_IdempotentsForA}. Up to the
naming of the indices, and after using eq. \ref{lBTA_18}, this is exactly our
claim in the case $p+1$ (and this is true despite the fact that we have only
considered $\delta^{\lbrack1]}$ so far -- because we shall next show that the
contribution from $H\circ\delta^{\lbrack2]}$ vanishes).\newline\textbf{(2)}
Consider $\delta_{q}^{[2]}$ in eq. \ref{lBT_CEComplexDifferential}. Using the
notation of eq. \ref{lBTA_15} we may write%
\[
\theta_{p+1,q}=%
{\textstyle\bigoplus\nolimits_{\deg(s_{1}\ldots s_{n})=p+1}}
{\textstyle\sum\nolimits_{\substack{w_{1}\ldots w_{p}\\\in\{1,\ldots
,n\}\text{,}\\\text{pairw. diff.}}}}
\theta_{p+1,q\mid s_{1}\ldots s_{n}}^{w_{1}\ldots w_{p}}\otimes f_{1}%
\wedge\widehat{f_{w_{1}}}\ldots\widehat{f_{w_{p}}}\wedge f_{n}%
\]
Therefore $\delta^{\lbrack2]}\theta_{p+1,q}$ equals%
\begin{align*}
&  =%
{\textstyle\bigoplus\nolimits_{\deg(s_{1}\ldots s_{n})=p+1}}
{\textstyle\sum\nolimits_{\substack{w_{1}\ldots w_{p}\\\in\{1,\ldots
,n\}\text{,}\\\text{pairw. diff.}}}}
{\textstyle\sum\nolimits_{\substack{w_{p+1}<w_{p+2}\\\in\{1,\ldots
,n\}\setminus\{w_{1},\ldots,w_{p}\}}}}
(-1)^{w_{p+1}+w_{p+2}}\\
&  (-1)^{\#\{w_{i}\mid1\leq i\leq p\text{ s.t. }w_{i}<w_{p+1}\}}%
(-1)^{\#\{w_{i}\mid1\leq i\leq p\text{ s.t. }w_{i}<w_{p+2}\}}\\
&  \theta_{p+1,q\mid s_{1}\ldots s_{n}}^{w_{1}\ldots w_{p}}\otimes\lbrack
f_{w_{p+1}},f_{w_{p+2}}]\wedge f_{1}\wedge\widehat{f_{w_{1}}}\ldots
\widehat{f_{w_{p+1}}}\ldots\widehat{f_{w_{p+2}}}\ldots\widehat{f_{w_{p}}%
}\wedge f_{n}\text{.}%
\end{align*}
The two terms $(-1)^{\#\{w_{i}\mid1\leq i\leq p\text{ s.t. }w_{i}<w_{p+1}\}}$
(and with $w_{i}<w_{p+2}$ respectively) appear since the original summand in
$\delta^{\lbrack2]}$ carries the sign $(-1)^{i+j}$, so we need to compute the
number of the wedge slot correctly, respecting the omitted wedge
factors;\ compare with the discussion in the first part of this proof. We
observe that the first wedge factor remains unchanged under $\delta
^{\lbrack2]}$. Hence, when we apply the contracting homotopy $H$ in this
induction step and in the next again, the summand will vanish thanks to
$H^{2}=0$, cf. eq. \ref{lBTA_35}. It will not do harm to verify this
explicitly: We use the notation of eq. \ref{lBTA_21} and write the above in
terms of%
\begin{align*}
(\delta^{\lbrack2]}\theta_{p+1,q})_{p+1,q-1\mid s_{1}\ldots s_{n}}%
^{w_{1}\ldots w_{p}\parallel w_{p+1},w_{p+2}}  &  =(-1)^{w_{p+1}+w_{p+2}%
}(-1)^{\#\{w_{i}\mid1\leq i\leq p\text{ s.t. }w_{i}<w_{p+1}\}}\\
&  (-1)^{\#\{w_{i}\mid1\leq i\leq p\text{ s.t. }w_{i}<w_{p+2}\}}%
\theta_{p+1,q\mid s_{1}\ldots s_{n}}^{w_{1}\ldots w_{p}}\text{.}%
\end{align*}
Next, we apply $H:N^{p+1}\rightarrow N^{p+2}$ (see eq. \ref{lBTA_17}):\ Then
for indices $s_{1}\ldots s_{n}=\gamma_{1}^{\dag}\ldots\gamma_{n-p-1}^{\dag
}0\ldots0$ and $\gamma_{1}^{\dag}\ldots\gamma_{n-p-1}^{\dag}\in\{\pm\}$ (which
is of degree $p+2$) we obtain the expression
\begin{align*}
&  (H\delta^{\lbrack2]}\theta_{p+1,q})_{p+2,q-1\mid\gamma_{1}^{\dag}%
\ldots\gamma_{n-p-1}^{\dag}\underset{p+1}{\underbrace{0\ldots0}}}^{w_{1}\ldots
w_{p}\parallel w_{p+1},w_{p+2}}\\
&  =P_{1}^{\gamma_{1}^{\dag}}\cdots P_{n-p-1}^{\gamma_{n-p-1}^{\dag}}%
{\textstyle\sum_{\gamma_{1},\ldots,\gamma_{n-p}\in\{\pm\}}}
(-1)^{(\ldots)}P_{n-p}^{-\gamma_{n-p}}\theta_{p+1,q\mid\gamma_{1}\cdots
\gamma_{n-p}\underset{p}{\underbrace{0\ldots0}}}^{w_{1}\ldots w_{p}}\text{,}%
\end{align*}
where we have plugged in our previous computation and started to disregard the
precise sign. We know the last term of this expression by our induction
hypothesis and therefore obtain%
\begin{align*}
&  =P_{1}^{\gamma_{1}^{\dag}}\cdots P_{n-p-1}^{\gamma_{n-p-1}^{\dag}}\\
&
{\textstyle\sum_{\gamma_{1},\ldots,\gamma_{n-p}\in\{\pm\}}}
{\textstyle\sum_{\gamma_{n-p+1}^{\ast}\ldots\gamma_{n}^{\ast}\in\{\pm\}}}
(-1)^{(\ldots)}\underline{P_{n-p}^{-\gamma_{n-p}}P_{1}^{\gamma_{1}}\cdots
P_{n-p}^{\gamma_{n-p}}}\\
&  \left(  P_{n-p+1}^{\left(  -\gamma_{n-p+1}^{\ast}\right)  }%
\operatorname*{ad}(f_{w_{p}})P_{n-p+1}^{\gamma_{n-p+1}^{\ast}}\right)
\cdots\left(  P_{n}^{\left(  -\gamma_{n}^{\ast}\right)  }\operatorname*{ad}%
(f_{w_{1}})P_{n}^{\gamma_{n}^{\ast}}\right)  f_{0}\text{.}%
\end{align*}
As the $P_{1}^{+},\ldots,P_{n}^{+}$ commute pairwise, the same holds for all
$P_{1}^{\pm},\ldots,P_{n}^{\pm}$ (by Def. \ref{BT_Def_IdempotentsForA}). Thus,
the underlined expression can be rearranged to $P_{n-p}^{-\gamma_{n-p}}%
P_{n-p}^{\gamma_{n-p}}\ldots$, but $P_{i}^{+}P_{i}^{-}=P_{i}^{+}%
(\mathbf{1}-P_{i}^{+})=0$ as $P_{i}^{+}$ is an idempotent. The same for
$P_{i}^{-}P_{i}^{+}$. Hence, in all the indices $s_{1}\ldots s_{n}$ relevant
for our claim $H\delta^{\lbrack2]}\theta_{p+1,q}$ is zero.
\end{proof}

This readily implies the following key computation:

\begin{theorem}
[Main Theorem]\label{BT_MainThmInner}Let $(A,(I_{i}^{\pm}),\tau)$ be an
$n$-fold cubically decomposed algebra over a field $k$. Then%
\begin{align*}
&  \left.  ^{\otimes}\operatorname*{res}\nolimits_{\ast}\right.  (f_{0}\otimes
f_{1}\wedge\ldots\wedge f_{n})=-(-1)^{\frac{(n-1)n}{2}}\\
&  \qquad\qquad\tau%
{\textstyle\sum_{\pi\in\mathfrak{S}_{n}}}
\operatorname*{sgn}(\pi)%
{\textstyle\sum_{\gamma_{1}\ldots\gamma_{n}\in\{\pm\}}}
\left(  -1\right)  ^{\gamma_{1}+\cdots+\gamma_{n}}(P_{1}^{-\gamma_{1}%
}\operatorname*{ad}f_{\pi(1)}P_{1}^{\gamma_{1}})\\
&  \qquad\qquad\cdots(P_{n}^{-\gamma_{n}}\operatorname*{ad}f_{\pi(n)}%
P_{n}^{\gamma_{n}})f_{0}\text{,}%
\end{align*}
where $P_{1}^{+},\ldots,P_{n}^{+}$ is any system of pairwise commuting good
idempotents in the sense of Def. \ref{BT_Def_IdempotentsForA} (the value does
not depend on the choice of the latter). Analogously,%
\[
(\left.  ^{\otimes}\operatorname*{res}\nolimits^{\ast}\right.  \varphi
)(f_{1}\wedge\ldots\wedge f_{n})(f_{0}):=\varphi\cdot\left.  ^{\otimes
}\operatorname*{res}\nolimits_{\ast}\right.  (f_{0}\otimes f_{1}\wedge
\ldots\wedge f_{n})
\]
for every $\varphi\in k$.
\end{theorem}

We remark that one can also write the above formula as%
\begin{align*}
&  \left.  ^{\otimes}\operatorname*{res}\nolimits_{\ast}\right.  (f_{0}\otimes
f_{1}\wedge\ldots\wedge f_{n})=-(-1)^{\frac{(n-1)n}{2}}\\
&  \tau%
{\textstyle\sum_{\pi\in\mathfrak{S}_{n}}}
\operatorname*{sgn}(\pi)%
{\textstyle\sum_{\gamma_{1}\ldots\gamma_{n}\in\{\pm\}}}
\left(  -1\right)  ^{\gamma_{1}+\cdots+\gamma_{n}}(P_{1}^{-\gamma_{1}}%
f_{\pi(1)}P_{1}^{\gamma_{1}})\cdots(P_{n}^{-\gamma_{n}}f_{\pi(n)}P_{n}%
^{\gamma_{n}})f_{0}%
\end{align*}
since for any expression $g$ we have%

\begin{align}
P_{i}^{-\gamma_{i}}\operatorname*{ad}(f_{w})P_{i}^{\gamma_{i}}g  &
=P_{i}^{-\gamma_{i}}[f_{w},P_{i}^{\gamma_{i}}g]=P_{i}^{-\gamma_{i}}f_{w}%
P_{i}^{\gamma_{i}}g-P_{i}^{-\gamma_{i}}P_{i}^{\gamma_{i}}gf_{w}\label{lBTA_33}%
\\
&  =P_{i}^{-\gamma_{i}}f_{w}P_{i}^{\gamma_{i}}g\nonumber
\end{align}
since $P_{i}^{-\gamma_{i}}P_{i}^{\gamma_{i}}=(\mathbf{1}-P_{i}^{\gamma_{i}%
})P_{i}^{\gamma_{i}}=0$ and $P_{i}^{\gamma_{i}}$ is an idempotent.

\begin{proof}
Use Prop. \ref{BT_KeyLemmaFormula} with $p=n$. Plugging these components into
the shorthand notation of eq. \ref{lBTA_15} we unwind for $\left.  ^{\otimes
}\operatorname*{res}\nolimits_{\ast}\right.  (f_{0}\otimes f_{1}\wedge
\ldots\wedge f_{n})$ the formula
\begin{align*}
&  =-\tau\left(  -1\right)  ^{\frac{n^{2}+n}{2}}\sum_{\substack{w_{1}\ldots
w_{n}\\=\{1,\ldots,n\}}}\rho(w_{1},\ldots,w_{n})(-1)^{w_{1}+\cdots+w_{n}}\\
&
{\textstyle\sum_{\gamma_{1}\ldots\gamma_{n}\in\{\pm\}}}
\left(  -1\right)  ^{\gamma_{1}+\cdots+\gamma_{n}}(P_{1}^{-\gamma_{1}%
}\operatorname*{ad}(f_{w_{n}})P_{1}^{\gamma_{1}})\cdots(P_{n}^{-\gamma_{n}%
}\operatorname*{ad}(f_{w_{1}})P_{n}^{\gamma_{n}})f_{0}\text{.}%
\end{align*}
We can clearly replace $w_{1},\ldots,w_{n}$ by a sum over all permutations of
$\{1,\ldots,n\}$. In order to obtain a nice formula (in the above formula the
$P_{i}$ appear in ascending order, while the $w_{i}$ appear in descending
order), we prefer to compose each permutation with the order-reversing
permutation $w_{i}:=\pi(n-i+1)$: Hence,%
\begin{align*}
&  =-\tau\left(  -1\right)  ^{\frac{n^{2}+n}{2}}\sum_{\pi\in\mathfrak{S}_{n}%
}\rho(\pi(n),\ldots,\pi(1))(-1)^{1+\cdots+n}\\
&
{\textstyle\sum_{\gamma_{1}\ldots\gamma_{n}\in\{\pm\}}}
\left(  -1\right)  ^{\gamma_{1}+\cdots+\gamma_{n}}(P_{1}^{-\gamma_{1}%
}\operatorname*{ad}(f_{\pi(1)})P_{1}^{\gamma_{1}})\cdots(P_{n}^{-\gamma_{n}%
}\operatorname*{ad}(f_{\pi(n)})P_{n}^{\gamma_{n}})f_{0}\text{.}%
\end{align*}
To conclude, use eq. \ref{lBTA_22} and the (easy)\ fact that the
order-reversing permutation has signum $(-1)^{\frac{(n-1)n}{2}}$, giving the
sign of our claim.
\end{proof}

\begin{proof}
[Proof of Thms. \ref{intro_Thm1UniversalTateCocycle} \&
\ref{intro_Thm2CocycleFormula}]We define $\mathfrak{G}:=E^{n}(k)$, where $E$
is the functor defined in \S \ref{TATE_section_InfiniteMatrixAlgebras}. As
already discussed in \S \ref{TATE_section_InfiniteMatrixAlgebras} this
contains $k[t_{1}^{\pm},\ldots,t_{n}^{\pm}]$ as a Lie subalgebra, acting as
multiplication operators $x\mapsto f\cdot x$. It is also easily checked that
the differential operators $t_{1}^{s_{1}}\cdots t_{n}^{s_{n}}\partial_{t_{i}}$
can be written as infinite matrices. If $\mathfrak{g}$ is a \textit{finite}%
-dimensional Lie algebra, observe that $\mathfrak{G}=E^{n}(k)$ and
$E^{n}(\operatorname*{End}_{k}(\mathfrak{g}))$ are actually isomorphic. If
$\mathfrak{g}$ is simple, it is centreless, so the adjoint representation
gives an embedding $\mathfrak{g}\hookrightarrow\operatorname*{End}%
\nolimits_{k}(\mathfrak{g})$, and thus%
\[
\mathfrak{g}[t_{1}^{\pm},\ldots,t_{n}^{\pm}]\hookrightarrow E^{n}%
(\operatorname*{End}\nolimits_{k}(\mathfrak{g}))\simeq E^{n}(k)=\mathfrak{G}%
\text{.}%
\]
This shows that all Lie algebras in the claim are subalgebras of
$\mathfrak{G}$. As shown in \S \ref{TATE_section_InfiniteMatrixAlgebras},
$\mathfrak{G}$ is a cubically decomposed algebra, so we define $\phi$ as in
Def. \ref{BT_Def_TateBeilinsonAbstractResidueMaps}, $\phi:=\left.
\operatorname*{res}\nolimits^{\ast}(1)\right.  $. Since we work with field
coefficients, the Universal Coefficient Theorem for Lie algebras tells us that%
\[
H^{n+1}(\mathfrak{g},k)\cong H_{n+1}(\mathfrak{g},k)^{\ast}\text{,}%
\]
i.e. knowing the values of a cocycle only on Lie cycles (instead of all of $%
{\textstyle\bigwedge\nolimits^{\bullet}}
\mathfrak{g}$) determines the cocycle uniquely, $\left.  \operatorname*{res}%
\nolimits^{\ast}(1)\right.  (\alpha)=\left.  \operatorname*{res}%
\nolimits_{\ast}\right.  \alpha$. However, by Lemma
\ref{BTrev_ComputeWedgeResViaTensorRes} we may evaluate the cocycle on the
image of $I$ by using $\left.  ^{\otimes}\operatorname*{res}\nolimits_{\ast
}\right.  $ instead. Using Thm. \ref{BT_MainThmInner} we get an explicit
formula for $\left.  ^{\otimes}\operatorname*{res}\nolimits^{\ast}(1)\right.
$, proving Thm. \ref{intro_Thm2CocycleFormula}. Using the explicit formula, it
is a direct computation to check that for $n=1$ the cocycle agrees with the
ones mentioned in the claim of Thm. \ref{intro_Thm1UniversalTateCocycle}.
\end{proof}

\section{\label{BT_sect_applications_residue}Application to the
Multidimensional Residue}

In this section we will show that the Lie cohomology class of Def.
\ref{BT_Def_TateBeilinsonAbstractResidueMaps} naturally gives the
multidimensional (Parshin) residue.

We work in the framework of multivariate Laurent polynomial rings over a field
$k$, see \S \ref{TATE_section_InfiniteMatrixAlgebras}. In other words, as our
cubically decomposed algebra we take an infinite matrix algebra $A=E^{n}(k)$
and $\mathfrak{g}=A_{Lie}$. Via eq. \ref{lBTA_30} it acts on the $k$-vector
space $k[t_{1}^{\pm},\ldots,t_{n}^{\pm}]$. The latter, now interpreted as a
ring, also embeds as a \textit{commutative} subalgebra into $A$. In order to
distinguish very clearly between the subalgebra of $A$ and the vector space it
acts on, we shall from now on write $k[\mathbf{t}_{1}^{\pm},\ldots
,\mathbf{t}_{n}^{\pm}]$ for the $k$-vector space. Thus, when we write $t_{i}$
we always refer to the associated multiplication operator $x\mapsto t_{i}\cdot
x$ in $A$, e.g. $t_{i}^{m}\cdot\mathbf{t}_{i}^{l}=\mathbf{t}_{i}^{m+l}%
$.\newline Following \cite[Lemma 1(b)]{MR565095} we may introduce a (not quite
well-defined\footnote{It does not respect the relation $\mathrm{d}%
(ab)=b\mathrm{d}a+a\mathrm{d}b$; this artifact already occurs in Beilinson's
paper \cite{MR565095}. However, this ambiguity dissolves after composing with
the residue (as in the theorem) and it is very convenient to treat this as
some sort of a map for the moment.}) `map'
\begin{equation}
\varkappa:\Omega_{k[t_{1}^{\pm},\ldots,t_{n}^{\pm}]/k}^{n}\rightarrow
H_{n+1}(\mathfrak{g},k)\qquad f_{0}\mathrm{d}f_{1}\wedge\ldots\wedge
\mathrm{d}f_{n}\mapsto f_{0}\wedge f_{1}\wedge\ldots\wedge f_{n}\text{.}
\label{lBT_36}%
\end{equation}
As $k[t_{1}^{\pm},\ldots,t_{n}^{\pm}]$ is commutative, the $f_{i}$ commute
pairwise and thus $f_{0}\wedge\ldots\wedge f_{n}$ is indeed a Lie homology cycle.

\begin{theorem}
\label{BT_PropDetFormulaForResidue}The morphism%
\[
\operatorname*{res}\nolimits_{\ast}\circ\varkappa:\Omega_{k[t_{1}^{\pm}%
,\ldots,t_{n}^{\pm}]/k}^{n}\longrightarrow k
\]
(with $\varkappa$ as in eq. \ref{lBT_36} and $\operatorname*{res}%
\nolimits_{\ast}$ as in Def. \ref{BT_Def_TateBeilinsonAbstractResidueMaps})
for $c_{i,j}\in\mathbf{Z}$ is explicitly given by%
\[
t_{1}^{c_{0,1}}\ldots t_{n}^{c_{0,n}}\mathrm{d}(t_{1}^{c_{1,1}}\ldots
t_{n}^{c_{1,n}})\wedge\ldots\wedge\mathrm{d}(t_{1}^{c_{n,1}}\ldots
t_{n}^{c_{n,n}})\mapsto-(-1)^{\frac{n^{2}+n}{2}}\det%
\begin{pmatrix}
c_{1,1} & \cdots & c_{n,1}\\
\vdots & \ddots & \vdots\\
c_{1,n} & \cdots & c_{n,n}%
\end{pmatrix}
\]
whenever $\sum_{p=0}^{n}c_{p,i}=0$ and is zero otherwise. In particular
$-(-1)^{\frac{n^{2}+n}{2}}(\operatorname*{res}\nolimits_{\ast}\circ\varkappa)$
is the conventional multidimensional (Parshin) residue.
\end{theorem}

The complicated sign $-(-1)^{\frac{n^{2}+n}{2}}$ should not concern us too
much; it is an artifact of homological algebra. Just by changing our sign
conventions for bicomplexes, we could easily switch to an overall opposite
sign. Letting $c_{i,j}=\delta_{i=j}$ for $i,j\in\{1,\ldots,n\}$ gives the
familiar%
\[
-(-1)^{\frac{n^{2}+n}{2}}\operatorname*{res}\nolimits_{\ast}(at_{1}^{c_{0,1}%
}\ldots t_{n}^{c_{0,n}}\wedge t_{1}\wedge\ldots\wedge t_{n})=\delta
_{c_{0,1}=-1}\cdots\delta_{c_{0,n}=-1}a
\]
for $a\in k$. In particular this assures us that the map $\operatorname*{res}%
\nolimits_{\ast}$ gives the correct notion of residue: it is the
$(-1,\ldots,-1)$-coefficient of the Laurent expansion.

\begin{proof}
After unwinding $\varkappa$ it remains to evaluate $\operatorname*{res}%
\nolimits_{\ast}(f_{0}\wedge f_{1}\wedge\ldots\wedge f_{n})$ for $f_{i}%
:=t_{1}^{c_{i,1}}\cdots t_{n}^{c_{i,n}}$ ($i=0,\ldots,n$). Clearly
$f_{0}\otimes f_{1}\wedge\ldots\wedge f_{n}$ is a cycle in $H_{n}%
(\mathfrak{g},\mathfrak{g})$, and so by Lemma
\ref{BTrev_ComputeWedgeResViaTensorRes} we may use $\left.  ^{\otimes
}\operatorname*{res}\nolimits_{\ast}\right.  $ instead of $\operatorname*{res}%
\nolimits_{\ast}$. Then Thm. \ref{BT_MainThmInner} reduces this to the matrix
trace%
\begin{align}
&  \operatorname*{res}\nolimits_{\ast}(f_{0}\wedge f_{1}\wedge\ldots\wedge
f_{n})=-(-1)^{\frac{(n-1)n}{2}}%
{\textstyle\sum_{\pi\in\mathfrak{S}_{n}}}
\operatorname*{sgn}(\pi)\tau M_{\pi}\text{, where}\label{lBT_38}\\
&  M_{\pi}:=%
{\textstyle\sum_{\gamma_{1}\ldots\gamma_{n}\in\{\pm\}}}
\left(  -1\right)  ^{\gamma_{1}+\cdots+\gamma_{n}}(P_{1}^{-\gamma_{1}}%
f_{\pi(1)}P_{1}^{\gamma_{1}})\cdots(P_{n}^{-\gamma_{n}}f_{\pi(n)}P_{n}%
^{\gamma_{n}})f_{0}\text{.}\nonumber
\end{align}
For the evaluation of $\tau M_{\pi}$ fix a permutation $\pi$ and pick the
(pairwise commuting) system of idempotents given by%
\begin{equation}
P_{j}^{+}\mathbf{t}_{1}^{\lambda_{1}}\cdots\mathbf{t}_{n}^{\lambda_{n}}%
=\delta_{\lambda_{j}\geq0}\mathbf{t}_{1}^{\lambda_{1}}\cdots\mathbf{t}%
_{n}^{\lambda_{n}}\text{.}\qquad\text{(with }\lambda_{1},\ldots,\lambda_{n}%
\in\mathbf{Z}\text{)} \label{lBTA_32}%
\end{equation}
Next, observe that the Laurent polynomial ring $W:=k[\mathbf{t}_{1}^{\pm
},\ldots,\mathbf{t}_{n}^{\pm}]$ is stable (i.e. $\phi W\subseteq W$) under the
endomorphisms $f_{0},\ldots,f_{n}$ and the idempotents $P_{i}^{\pm}$, and
therefore under $M_{\pi}$. Hence, it follows that it suffices to evaluate the
trace of $M_{\pi}$ on the $k$-vector subspace $k[\mathbf{t}_{1}^{\pm}%
,\ldots,\mathbf{t}_{n}^{\pm}]$. We compute successively%
\begin{align*}
f_{k}P_{j}^{+}\mathbf{t}_{1}^{\lambda_{1}}\cdots\mathbf{t}_{n}^{\lambda_{n}}
&  =\delta_{\lambda_{j}\geq0}\mathbf{t}_{1}^{\lambda_{1}+c_{k,1}}%
\cdots\mathbf{t}_{n}^{\lambda_{n}+c_{k,n}}\\
P_{j}^{-}f_{k}P_{j}^{+}\mathbf{t}_{1}^{\lambda_{1}}\cdots\mathbf{t}%
_{n}^{\lambda_{n}}  &  =\delta_{0\leq\lambda_{j}<-c_{k,j}}\mathbf{t}%
_{1}^{\lambda_{1}+c_{k,1}}\cdots\mathbf{t}_{n}^{\lambda_{n}+c_{k,n}}%
\end{align*}
and analogously for $P_{j}^{+}f_{k}P_{j}^{-}$. We find%
\begin{align}
&
{\textstyle\sum\nolimits_{\gamma_{j}\in\{\pm\}}}
(-1)^{\gamma_{j}}\left(  P_{j}^{-\gamma_{j}}f_{k}P_{j}^{\gamma_{j}}\right)
\mathbf{t}_{1}^{\lambda_{1}}\cdots\mathbf{t}_{n}^{\lambda_{n}}\label{lBT_28}\\
&  \qquad=(\delta_{0\leq\lambda_{j}<-c_{k,j}}-\delta_{-c_{k,j}\leq\lambda
_{j}<0})\mathbf{t}_{1}^{\lambda_{1}+c_{k,1}}\cdots\mathbf{t}_{n}^{\lambda
_{n}+c_{k,n}}\text{.}\nonumber
\end{align}
Now we claim:\medskip

\begin{itemize}
\item Subclaim: \textit{Writing }$w_{i}:=\pi(i)$ \textit{we have}%
\begin{align}
M_{\pi}\mathbf{t}_{1}^{\lambda_{1}}\cdots\mathbf{t}_{n}^{\lambda_{n}}  &  =%
{\textstyle\prod\limits_{i=1}^{n}}
(\delta_{0\leq\lambda_{i}+c_{0,i}+\sum_{p=i+1}^{n}c_{w_{p},i}<-c_{w_{i},i}%
}\nonumber\\
&  -\delta_{-c_{w_{i},i}\leq\lambda_{i}+c_{0,i}+\sum_{p=i+1}^{n}c_{w_{p},i}%
<0})\nonumber\\
&  \mathbf{t}_{1}^{\lambda_{1}+c_{0,1}+\sum_{p=1}^{n}c_{w_{p},1}}%
\cdots\mathbf{t}_{n}^{\lambda_{n}+c_{0,n}+\sum_{p=1}^{n}c_{w_{p},n}}\text{.}
\label{lBT_37}%
\end{align}
\medskip
\end{itemize}

(\textit{Proof:} Define for $i=1,\ldots,n+1$ the truncated sum%
\[
M_{\pi}^{(i)}:=\left[
{\textstyle\sum_{\gamma_{i}\ldots\gamma_{n}\in\{\pm\}}}
\left(  -1\right)  ^{\gamma_{i}+\cdots+\gamma_{n}}(P_{i}^{-\gamma_{i}}%
f_{w_{i}}P_{i}^{\gamma_{i}})\cdots(P_{n}^{-\gamma_{n}}f_{w_{n}}P_{n}%
^{\gamma_{n}})\right]  f_{0}%
\]
so that $M_{\pi}^{(1)}=M_{\pi}$ and $M_{\pi}^{(n+1)}=f_{0}$. We claim that%
\begin{equation}
M_{\pi}^{(i)}\mathbf{t}_{1}^{\lambda_{1}}\cdots\mathbf{t}_{n}^{\lambda_{n}%
}=\alpha\mathbf{t}_{1}^{\lambda_{1}+c_{0,1}+\sum_{p=i}^{n}c_{w_{p},1}}%
\cdots\mathbf{t}_{n}^{\lambda_{n}+c_{0,n}+\sum_{p=i}^{n}c_{w_{p},n}}
\label{lBT_29}%
\end{equation}
for some factor $\alpha\in\{\pm1,0\}$. For $i=n+1$ this is clear since
$f_{0}=t_{1}^{c_{0,1}}\cdots t_{n}^{c_{0,n}}$, in particular $\alpha=1$.
Assuming this holds for $i+1$, for $i$ we get by using eq. \ref{lBT_28} (with
the appropriate values plugged in: $j:=i$ and $k:=w_{i}$, and $\lambda_{i}$ as
in eq. \ref{lBT_29})%
\begin{align}
M_{\pi}^{(i)}\mathbf{t}_{1}^{\lambda_{1}}\cdots\mathbf{t}_{n}^{\lambda_{n}}
&  =%
{\textstyle\sum\nolimits_{\gamma_{i}\in\{\pm\}}}
(-1)^{\gamma_{i}}\left(  P_{i}^{-\gamma_{i}}f_{w_{i}}P_{i}^{\gamma_{i}%
}\right)  M_{\pi}^{(i+1)}\mathbf{t}_{1}^{\lambda_{1}}\cdots\mathbf{t}%
_{n}^{\lambda_{n}}\nonumber\\
&  \mathbf{=}(\delta_{0\leq\lambda_{i}+c_{0,i}+\sum_{p=i+1}^{n}c_{w_{p}%
,i}<-c_{w_{i},i}}-\delta_{-c_{w_{i},i}\leq\lambda_{i}+c_{0,i}+\sum_{p=i+1}%
^{n}c_{w_{p},i}<0})\label{lBT_31}\\
&  \alpha\mathbf{t}_{1}^{\lambda_{1}+c_{0,1}+\sum_{p=i+1}^{n}c_{w_{p}%
,1}+c_{w_{i},1}}\cdots\mathbf{t}_{n}^{\lambda_{n}+c_{0,n}+\sum_{p=i+1}%
^{n}c_{w_{p},n}+c_{w_{i},n}}\text{.}\nonumber
\end{align}
This proves our claim for all $i$ by induction. We observe that the pre-factor
$\alpha$ in each step just gets multiplied with the expression is eq.
\ref{lBT_31}, giving the product in our claim.\medskip)

Next, we need to evaluate the trace of $M_{\pi}$ as given in eq. \ref{lBT_37}.
The endomorphism is nilpotent unless%
\begin{equation}
\forall i:c_{0,1}+%
{\textstyle\sum\nolimits_{p=1}^{n}}
c_{w_{p},i}=0\text{.} \label{lBT_23}%
\end{equation}
We remark that $w_{1},\ldots,w_{n}$ is just a permutation of $\{1,\ldots,n\}$,
so these conditions can be rewritten as $\sum_{p=0}^{n}c_{p,i}=0$. In the
nilpotent case the trace is clearly zero. Hence, we may assume we are in the
case where eq. \ref{lBT_23} holds. Using these equations and the useful
convention $w_{n+1}:=0$, our expression for $M_{\pi}$ simplifies to%
\begin{align}
M_{\pi}\mathbf{t}_{1}^{\lambda_{1}}\cdots\mathbf{t}_{n}^{\lambda_{n}}  &  =%
{\textstyle\prod\limits_{i=1}^{n}}
(\delta_{0\leq\lambda_{i}+\sum_{p=i+1}^{n+1}c_{w_{p},i}<-c_{w_{i},i}%
}\label{lBT_27}\\
&  -\delta_{0\leq\lambda_{i}+c_{w_{i},i}+\sum_{p=i+1}^{n+1}c_{w_{p}%
,i}<c_{w_{i},i}})\mathbf{t}_{1}^{\lambda_{1}}\cdots\mathbf{t}_{n}^{\lambda
_{n}}\text{.}\nonumber
\end{align}
The endomorphism $M_{\pi}$ is visibly diagonal of finite rank and we may
reduce the computation of the trace to a (finite-dimensional) stable vector
subspace. A finite subset of the $\mathbf{t}_{1}^{\lambda_{1}}\cdots
\mathbf{t}_{n}^{\lambda_{n}}$ ($\lambda_{1},\ldots,\lambda_{n}\in\mathbf{Z}$)
provides a basis. We see in eq. \ref{lBT_27} that $M_{\pi}$ acts diagonally on
these basis vectors with eigenvalues $\pm1$ or $0$. Moreover, for each $i$ we
either have $c_{w_{i},i}\geq0$ or $c_{w_{i},i}<0$, which shows that each
bracket of the shape $(\delta_{0\leq\lambda<-c}-\delta_{-c\leq\lambda<0})$ in
eq. \ref{lBT_27} either attains only values in $\{+1,0\}$ when we run through
all $\lambda_{1},\ldots,\lambda_{n}\in\mathbf{Z}$, or only values in
$\{-1,0\}$. This shows that we only need to count (with appropriate sign) the
non-zero eigenvalues of $M_{\pi}$ in order to evaluate the trace. Note that
our finite subset of $\mathbf{t}_{1}^{\lambda_{1}}\cdots\mathbf{t}%
_{n}^{\lambda_{n}}$ ($\lambda_{1},\ldots,\lambda_{n}\in\mathbf{Z}$) indexes a
basis, so we need to count the number of such basis vectors with non-zero
eigenvalue. We introduce the non-standard shorthand $\left\lfloor
x\right\rfloor :=\min(0,x)$. Inspecting eq. \ref{lBT_27} shows that when
running through $\lambda_{i}$ we have

\begin{itemize}
\item $\left\lfloor -c_{w_{i},i}\right\rfloor $ times the eigenvalue $+1$,

\item $\left\lfloor +c_{w_{i},i}\right\rfloor $ times the eigenvalue $-1$.
\end{itemize}

The value of a fixed bracket $(\delta_{0\leq\lambda<-c}-\delta_{-c\leq
\lambda<0})$ - when non-zero - is always either $+1$, or always $-1$. Thus,
the number of non-zero eigenvalues is simply the number of elements within the
hypercube such that each $\lambda_{i}$ lies within the range of length
$\left\lfloor \pm c_{w_{i},i}\right\rfloor $ counted above, and therefore%
\[
\tau M_{\pi}=%
{\textstyle\prod\nolimits_{i=1}^{n}}
(\left\lfloor -c_{w_{i},i}\right\rfloor -\left\lfloor +c_{w_{i},i}%
\right\rfloor )=%
{\textstyle\prod\nolimits_{i=1}^{n}}
(-c_{w_{i},i})=(-1)^{n}%
{\textstyle\prod\nolimits_{i=1}^{n}}
c_{\pi(i),i}%
\]
(because $\left\lfloor -a\right\rfloor -\left\lfloor a\right\rfloor =-a$ for
all $a\in\mathbf{Z}$). We plug this into eq. \ref{lBT_38} and recognize the
usual formula for the determinant. This finishes the proof.
\end{proof}

We are now ready to prove the remaining theorems from the introduction:

\begin{proof}
[Proof of Thms. \ref{Prop_MainResidueThm} \&
\ref{intro_Thm5CocycleAgreesWithBeilinsons}]We use Thm.
\ref{BT_PropDetFormulaForResidue} to obtain Thm. \ref{Prop_MainResidueThm}%
.\ref{resthm_part2}. Then Thm. \ref{Prop_MainResidueThm}.\ref{resthm_part3}
follows as a special case. For Thm. \ref{Prop_MainResidueThm}%
.\ref{resthm_part1} use the shorthands $\pi=P_{1}^{+}=P^{+}$ (following both
the notation of Arbarello, de Concini and Kac and ours). On the one hand we
compute%
\begin{align*}
\lbrack\pi,f_{1}]f_{0}  &  =[P,f_{1}]f_{0}=Pf_{1}f_{0}-f_{1}Pf_{0}%
=[Pf_{0},f_{1}]\\
&  =(P^{+}+P^{-})[P^{+}f_{0},f_{1}]=P^{-}[P^{+}f_{0},f_{1}]+P^{+}[P^{+}%
f_{0},f_{1}]
\end{align*}
and we have $[P^{+}f_{0},f_{1}]+[P^{-}f_{0},f_{1}]=[f_{0},f_{1}]=0$, so this
equals%
\[
=P^{-}[P^{+}f_{0},f_{1}]-P^{+}[P^{-}f_{0},f_{1}]\text{.}%
\]
On the other hand, we unwind%
\begin{align*}
\operatorname*{res}f_{0}\mathrm{d}f_{1}  &  =\left(  -1\right)  ^{1}%
\operatorname*{tr}%
{\textstyle\sum_{\gamma_{1}\in\{\pm\}}}
\left(  -1\right)  ^{\gamma_{1}}(P_{1}^{-\gamma_{1}}\operatorname*{ad}%
(f_{\pi(1)})P_{1}^{\gamma_{1}})f_{0}\\
&  =-P^{-}[f_{1},P_{1}^{+}f_{0}]+P^{+}[f_{1},P_{1}^{-}f_{0}]
\end{align*}
and these expressions clearly coincide. Finally Thm.
\ref{intro_Thm5CocycleAgreesWithBeilinsons} is true since we use the cocycle
defined in Def. \ref{BT_Def_TateBeilinsonAbstractResidueMaps}, i.e. it is
constructed exactly as stated in Thm.
\ref{intro_Thm5CocycleAgreesWithBeilinsons}.
\end{proof}

\section{\label{BT_sect_applications_multiloop}Application to Multiloop Lie
Algebras}

Suppose $k$ is a field and $\mathfrak{g}/k$ is a finite-dimensional centreless
Lie algebra (e.g. $\mathfrak{g}$ finite-dimensional, semisimple). Then the
adjoint representation $\operatorname*{ad}:\mathfrak{g}\hookrightarrow
\operatorname*{End}\nolimits_{k}(\mathfrak{g})$ is injective. Thus, we obtain
a Lie algebra inclusion%

\[
i:\mathfrak{g}[\mathbf{t}_{1}^{\pm},\ldots,\mathbf{t}_{n}^{\pm}%
]\hookrightarrow E^{n}(\operatorname*{End}\nolimits_{k}(\mathfrak{g}%
))_{Lie}\text{,}%
\]
where $E$ is the functor described in
\S \ref{TATE_section_InfiniteMatrixAlgebras} (the right-hand side is equipped
with the Lie bracket $[a,b]=ab-ba$ based on the associative algebra
structure). Thus, we have the pullback%
\[
i^{\ast}:H^{n+1}(E^{n}(\operatorname*{End}\nolimits_{R}(\mathfrak{g}%
))_{Lie},k)\rightarrow H^{n+1}(\mathfrak{g}[\mathbf{t}_{1}^{\pm}%
,\ldots,\mathbf{t}_{n}^{\pm}],k)\text{,}%
\]
which we may apply to the class $\operatorname*{res}\nolimits^{\ast}(1)$, see
Def. \ref{BT_Def_TateBeilinsonAbstractResidueMaps}.

\begin{theorem}
Suppose $k$ is a field and $\mathfrak{g}/k$ is a finite-dimensional centreless
Lie algebra. For $Y_{0},\ldots,Y_{n}\in\mathfrak{g}$ we call%
\begin{equation}
B(Y_{0},\ldots,Y_{n}):=\operatorname*{tr}\nolimits_{\operatorname*{End}%
_{k}(\mathfrak{g})}(\operatorname*{ad}(Y_{0})\operatorname*{ad}(Y_{1}%
)\cdots\operatorname*{ad}(Y_{n})) \label{lTateLieCase_GeneralizedKillingForm}%
\end{equation}
the `generalized Killing form'. For $n=1$ and if $\mathfrak{g}$ is semisimple,
this is the classical Killing form of $\mathfrak{g}$.

\begin{enumerate}
\item Then on all Lie cycles admitting a lift under $I$ as in eq.
\ref{lBT_revIIntro}, the pullback~$i^{\ast}\operatorname*{res}\nolimits^{\ast
}(1)\in H^{n+1}(\mathfrak{g}[\mathbf{t}_{1}^{\pm},\ldots,\mathbf{t}_{n}^{\pm
}],k)$ is explicitly given by%
\begin{align*}
&  (i^{\ast}\phi)(Y_{0}\mathbf{t}_{1}^{c_{0,1}}\cdots\mathbf{t}_{n}^{c_{0,n}%
}\wedge\cdots\wedge Y_{n}\mathbf{t}_{1}^{c_{n,1}}\cdots\mathbf{t}_{n}%
^{c_{n,n}})\\
&  =-(-1)^{\frac{n^{2}+n}{2}}\sum_{\pi\in\mathfrak{S}_{n}}\operatorname*{sgn}%
(\pi)B(Y_{\pi(1)},\ldots,Y_{\pi(n)},Y_{0})\prod\limits_{i=1}^{n}c_{\pi
(i),i}\text{.}%
\end{align*}
whenever $\forall i\in\{1,\ldots,n\}:\sum_{p=0}^{n}c_{p,i}=0$ and zero otherwise.

\item If $\mathfrak{g}$ is finite-dimensional and semisimple and $n=1$, then
$i^{\ast}\operatorname*{res}\nolimits^{\ast}(1)\in H^{2}(\mathfrak{g}%
[\mathbf{t}_{1}^{\pm}],k)$ is the universal central extension of the loop Lie
algebra $\mathfrak{g}[\mathbf{t}_{1},\mathbf{t}_{1}^{-1}]$ giving the
associated affine Lie algebra $\widehat{\mathfrak{g}}$ (without extending by a
derivation),%
\[
0\longrightarrow k\left\langle c\right\rangle \longrightarrow\widehat
{\mathfrak{g}}\longrightarrow\mathfrak{g}[\mathbf{t}_{1},\mathbf{t}_{1}%
^{-1}]\longrightarrow0\text{.}%
\]

\end{enumerate}
\end{theorem}

\begin{proof}
\textbf{(1)} According to Lemma \ref{BTrev_ComputeWedgeResViaTensorRes}, Thm.
\ref{BT_MainThmInner} and eq. \ref{lBT_39} the cocycle is explicitly given by%
\begin{align*}
\operatorname*{res}\nolimits^{\ast}(1)(f_{0}\wedge\cdots\wedge f_{n})  &
=\left.  ^{\otimes}\operatorname*{res}\nolimits^{\ast}\right.  (1)(f_{0}%
\otimes f_{1}\wedge\cdots\wedge f_{n})\\
&  =\tau%
{\textstyle\sum_{\pi\in\mathfrak{S}_{n}}}
\operatorname*{sgn}(\pi)M_{\pi}\text{, where}\\
M_{\pi}  &  =%
{\textstyle\sum_{\gamma_{1}\ldots\gamma_{n}\in\{\pm\}}}
\left(  -1\right)  ^{\gamma_{1}+\cdots+\gamma_{n}}\\
&  (P_{1}^{-\gamma_{1}}f_{\pi(1)}P_{1}^{\gamma_{1}})\cdots(P_{n}^{-\gamma_{n}%
}f_{\pi(n)}P_{n}^{\gamma_{n}})f_{0}\text{.}%
\end{align*}
Note that $M_{\pi}\in E^{n}(\operatorname*{End}\nolimits_{k}(\mathfrak{g}))$.
As we consider the pullback of the cohomology class along $i:\mathfrak{g}%
[t_{1}^{\pm},\ldots,t_{n}^{\pm}]\hookrightarrow E^{n}(\operatorname*{End}%
\nolimits_{k}(\mathfrak{g}))_{Lie}$, it suffices to treat elements
$f_{i}:=Y_{i}t_{1}^{c_{i,1}}\cdots t_{n}^{c_{i,n}}$ with $c_{i,1}%
,\ldots,c_{i,n}\in\mathbf{Z}$ (for $i=0,\ldots,n$) and $Y_{i}\in\mathfrak{g}$.
Note that by our embedding $i$ an element $f_{i}$ is mapped to the
endomorphism $\operatorname*{ad}(Y_{i})t_{1}^{c_{i,1}}\cdots t_{n}^{c_{i,n}}$
in $E^{n}(\operatorname*{End}\nolimits_{k}(\mathfrak{g}))$. Let $\pi
\in\mathfrak{S}_{n}$ be a fixed permutation. In order to compute the trace, it
suffices to study the action of $M_{\pi}$ on the basis elements $X\mathbf{t}%
_{1}^{\lambda_{1}}\cdots\mathbf{t}_{n}^{\lambda_{n}}$ of $\mathfrak{g}%
[t_{1}^{\pm},\ldots,t_{n}^{\pm}]$, where $\lambda_{1},\ldots,\lambda_{n}%
\in\mathbf{Z}$ and $X\in\mathfrak{g}$ runs through a basis of $\mathfrak{g}$.
We denote them with bold letters $\mathbf{t}_{i}$ instead of $t_{i}$ to
distinguish clearly between a basis element and $t_{i}$ as an endomorphism
$t_{i}:x\mapsto t_{i}\cdot x$ in $E^{n}(\operatorname*{End}\nolimits_{k}%
(\mathfrak{g}))$. As in the proof of Thm. \ref{BT_PropDetFormulaForResidue} we
compute%
\[
P_{j}^{-}f_{k}P_{j}^{+}X\mathbf{t}_{1}^{\lambda_{1}}\cdots\mathbf{t}%
_{n}^{\lambda_{n}}=\delta_{0\leq\lambda_{j}<-c_{k,j}}\operatorname*{ad}%
(Y_{k})X\mathbf{t}_{1}^{\lambda_{1}+c_{k,1}}\cdots\mathbf{t}_{n}^{\lambda
_{n}+c_{k,n}}\text{.}%
\]
and as a consequence we find%
\begin{align*}
&
{\textstyle\sum\nolimits_{\gamma_{j}\in\{\pm\}}}
\left(  -1\right)  ^{\gamma_{j}}(P_{j}^{-\gamma_{j}}x_{k}P_{j}^{\gamma_{j}%
})X\mathbf{t}_{1}^{\lambda_{1}}\cdots\mathbf{t}_{n}^{\lambda_{n}}\\
&  \qquad=(\delta_{0\leq\lambda_{j}<-c_{k,j}}-\delta_{-c_{k,j}\leq\lambda
_{j}<0})\operatorname*{ad}(Y_{k})X\mathbf{t}_{1}^{\lambda_{1}+c_{k,1}}%
\cdots\mathbf{t}_{n}^{\lambda_{n}+c_{k,n}}\text{.}%
\end{align*}
With an inductive computation entirely analogous to eq. \ref{lBT_37} we find%
\begin{align*}
M_{\pi}X\mathbf{t}_{1}^{\lambda_{1}}\cdots\mathbf{t}_{n}^{\lambda_{n}}  &  =%
{\textstyle\prod\nolimits_{i=1}^{n}}
(\delta_{0\leq\lambda_{i}+c_{0,i}+%
{\textstyle\sum\nolimits_{p=i+1}^{n}}
c_{w_{p},i}<-c_{w_{i},i}}\\
&  -\delta_{-c_{w_{i},i}\leq\lambda_{i}+c_{0,i}+%
{\textstyle\sum\nolimits_{p=i+1}^{n}}
c_{w_{p},i}<0})\\
&  \qquad\operatorname*{ad}(Y_{w_{1}})\cdots\operatorname*{ad}(Y_{w_{n}%
})\operatorname*{ad}(Y_{0})X\\
&  \mathbf{t}_{1}^{\lambda_{1}+\sum\nolimits_{p=0}^{n}c_{p,1}}\cdots
\mathbf{t}_{n}^{\lambda_{n}+\sum\nolimits_{p=0}^{n}c_{p,n}}\text{,}%
\end{align*}
where $w_{i}:=\pi(i)$. Unless $\forall i:%
{\textstyle\sum\nolimits_{p=0}^{n}}
c_{p,i}=0$ holds, $M_{\pi}$ is clearly nilpotent and thus has trace $\tau
M_{\pi}=0$. This condition is clearly independent of $\pi$, showing that
$(i^{\ast}\operatorname*{res}\nolimits^{\ast}(1))(f_{0}\wedge\cdots\wedge
f_{n})=0$ in this case. From now on assume $\forall i:%
{\textstyle\sum\nolimits_{p=0}^{n}}
c_{p,i}=0$. Then $M_{\pi}$ respects the decomposition%
\[
\mathfrak{g}[\mathbf{t}_{1}^{\pm},\ldots,\mathbf{t}_{n}^{\pm}]=%
{\textstyle\coprod\nolimits_{\lambda_{1},\ldots,\lambda_{n}\in\mathbf{Z}^{n}}}
\mathfrak{g}\mathbf{t}_{1}^{\lambda_{1}}\cdots\mathbf{t}_{n}^{\lambda_{n}}%
\]
and therefore (as $\tau$ is essentially a trace) $\tau M_{\pi}=%
{\textstyle\sum_{\lambda_{1},\ldots,\lambda_{n}}}
\tau M_{\pi}\mid_{\mathfrak{g}\mathbf{t}_{1}^{\lambda_{1}}\cdots\mathbf{t}%
_{n}^{\lambda_{n}}}$.\ For each summand of the latter we obtain%
\begin{align*}
\tau M_{\pi}\mid_{\mathfrak{g}\mathbf{t}_{1}^{\lambda_{1}}\cdots\mathbf{t}%
_{n}^{\lambda_{n}}}  &  =%
{\textstyle\prod\nolimits_{i=1}^{n}}
(\delta_{0\leq\lambda_{i}+c_{0,i}+%
{\textstyle\sum\nolimits_{p=i+1}^{n}}
c_{w_{p},i}<-c_{w_{i},i}}\\
&  \qquad-\delta_{-c_{w_{i},i}\leq\lambda_{i}+c_{0,i}+%
{\textstyle\sum\nolimits_{p=i+1}^{n}}
c_{w_{p},i}<0})\\
&  \operatorname*{tr}(\operatorname*{ad}(Y_{w_{1}})\cdots\operatorname*{ad}%
(Y_{w_{n}})\operatorname*{ad}(Y_{0}))\text{.}%
\end{align*}
The trace term is independent of $\lambda_{1},\ldots,\lambda_{n}$ (and in the
shape of eq. \ref{lTateLieCase_GeneralizedKillingForm}), so we may rewrite
$\tau M_{\pi}$ as%
\begin{align*}
\tau M_{\pi}  &  =B(Y_{w_{1}},\ldots,Y_{w_{n}},Y_{0})%
{\textstyle\sum\nolimits_{\lambda_{1},\ldots,\lambda_{n}}}
{\textstyle\prod\nolimits_{i=1}^{n}}
(\delta_{0\leq\lambda_{i}+c_{0,i}+%
{\textstyle\sum\nolimits_{p=i+1}^{n}}
c_{w_{p},i}<-c_{w_{i},i}}\\
&  -\delta_{-c_{w_{i},i}\leq\lambda_{i}+c_{0,i}+%
{\textstyle\sum\nolimits_{p=i+1}^{n}}
c_{w_{p},i}<0})\text{.}%
\end{align*}
For the evaluation of the sum $%
{\textstyle\sum\nolimits_{\lambda_{1},\ldots,\lambda_{n}}}
$ we can apply the same eigenvalue count as in the proof of Thm.
\ref{BT_PropDetFormulaForResidue}.\ This time instead of counting eigenvalues,
we count non-zero summands. This yields%
\[
\tau M_{\pi}=(-1)^{n}B(Y_{w_{1}},\ldots,Y_{w_{n}},Y_{0})%
{\textstyle\prod\nolimits_{i=1}^{n}}
c_{w_{i},i}%
\]
and thus our claim. \textbf{(2)} For $n=1$ we obtain
\[
(i^{\ast}\operatorname*{res}\nolimits^{\ast}(1))(Y_{0}\mathbf{t}_{1}^{c_{0,1}%
}\wedge Y_{1}\mathbf{t}_{1}^{c_{1,1}})=-c_{1,1}\delta_{c_{0,1}+c_{1,1}%
=0}B(Y_{1},Y_{0})\text{.}%
\]
This is well-known to be the defining cocycle of the affine Lie algebra
$\widehat{\mathfrak{g}}$ (usually with a positive sign, but the class is only
well-defined up to non-zero scalar multiple anyway).
\end{proof}

The natural further cases of the Virasoro algebra as well as affine Kac-Moody
algebras (i.e. $\widehat{\mathfrak{g}}$ extended by derivations) will be
discussed elsewhere. The computations become more involved, but no further
ideas are needed.

\addcontentsline{toc}{chapter}{References}
\bibliographystyle{amsplain}
\bibliography{Ktheory,ollinewbib}

\end{document}